\documentclass[reqno,11pt]{amsart}

\usepackage{amsmath,amsfonts,amssymb,amsthm,epsfig,enumerate}
\usepackage{xcolor}
\usepackage[english]{babel}
\usepackage{hyperref} 
 \allowdisplaybreaks
\numberwithin{equation}{section}

\font\script=rsfs10 at 11pt
\def\eps{\varepsilon}
\def\H{{\mbox{\script H}\,\,}}

\def\E{\mathcal E}
\def\F{\mathcal F}
\def\G{\mathcal G}

\def\Fot{F_j^{\rm \,other}}
\def\R{\mathbb R}
\def\S{\mathbb S}
\def\N{\mathbb N}
\def\diam{{\rm diam}}
\def\jm{{j_{\rm max}}}
\def\freccia#1{\xrightarrow[\ #1]{}}
\def\tv#1{\big| |#1|\big|}
\def\bbeps{\overline{\overline\eps}}
\def\bal{\begin{aligned}}
\def\eal{\end{aligned}}
\def\proofof#1{\begin{proof}[Proof of #1]}
\def\Chi#1{\hbox{{\large $\chi$}{\Large $_{_{#1}}$}}}
\def\step#1#2{\par\noindent{\underline{\it Step~#1.}}\emph{ #2}\\}
\def\comp{\subset\subset}

\newcounter{mt}
\def\maintheorem#1#2#3{\par \medskip \noindent {\bf Theorem~\mref{#1}}~(#2).~{\it #3}\par}
\def\mref#1{\Alph{#1}}
\def\maintheoremdeclaration#1{\stepcounter{mt}\newcounter{#1}\setcounter{#1}{\arabic{mt}}}

\maintheoremdeclaration{main1}
\maintheoremdeclaration{main2}

\newtheorem{theorem}{Theorem}[section]
\newtheorem{lemma}[theorem]{Lemma}
\newtheorem{prop}[theorem]{Proposition}
\newtheorem{corol}[theorem]{Corollary}
\newtheorem{defin}[theorem]{Definition}
\newtheorem{remark}[theorem]{Remark}

\begin{document}

\title[Small minimal clusters in manifolds]{Connectedness properties of small minimal clusters in Riemannian or Finsler manifolds}

\author[S.~Nardulli]{S.~Nardulli}
\address[S.~Nardulli]{Universidade Federal do ABC, Centro de Matem\'atica, Computa\c c\~ao e Cogni\c c\~ao, Avenida dos Estados 5001, S\~ao Paulo, Brasil}
\email{stefano.nardulli@ufabc.edu.br}
\author[A.~Pratelli]{A.~Pratelli}
\address[A.~Pratelli]{Dipartimento di Matematica, Universit\`a di Pisa, Largo B. Pontecorvo 5, 56127 Pisa, Italy}
\email{aldo.pratelli@unipi.it}

\begin{abstract}
We prove that in a compact Riemannian manifold, the $m$-minimal clusters of sufficiently small total volume are connected and have small diameter, while in a more general Finsler manifold they are made up of at most $m$ connected components of small diameter. We apply these results to calculate the asymptotic expansion of the multi-isoperimetric profile at the first nontrivial order, for small volumes.
\end{abstract}

\maketitle

\section{Introduction}

We are interested in the isoperimetric problem for clusters in a Riemannian or a Finsler manifold, in particular in the connectedness properties of small minimal clusters. In fact, while it is clear that a generic minimal cluster could be made up of an arbitrary number of connected components, the situation changes when considering minimal clusters of \emph{small} volume. Indeed, since the manifold locally looks like $\R^N$ with a fixed convex norm (and in particular, the norm is the Euclidean one if the manifold is of Riemannian type), one can guess that minimal clusters of small volume behave more or less like minimal clusters in situations with a fixed norm, which are well known to be connected.

The situation for isoperimetric sets, which correspond to $m$-clusters for $m=1$, is already quite well known; see for instance~\cite{FigalliMaggi2011,drn,NardulliAGAG09,MR1803220} and the references therein. In particular, in~\cite{FigalliMaggi2011} it is shown that small isoperimetric sets are connected, while in the recent paper~\cite{drn} it is shown that, in the case of a Riemannian manifold, they are also geodesically convex and quantitatively close to the corresponding Wulff shapes.

What we are able to show in this paper is that sufficiently small minimal $m$-clusters are actually connected in a compact Riemannian manifold, and they are made up of at most $m$ connected components for a compact Finsler manifold. In particular, this gives the connectedness of small isoperimetric sets (i.e., $m$-clusters with $m=1$) also in the case of a Finsler manifold. Our main result for Finsler manifolds (the relative definitions are given in Section~\ref{notdef}) can be stated as follows.

\maintheorem{main1}{Connectedness: the Finsler case}{
Let $M$ be an $N$-dimensional, ${\rm C}^2$, compact Finsler manifold (not necessarily reversible), and let $m\in\N$ be given. There exist two constants $\bar\eps>0$ and $K>0$ such that any minimal $m$-cluster $\E$ with total volume $\tv{\E}<\bar\eps$ is made up of at most $m$ connected components $\E_i$. In addition, each connected component satisfies
\begin{equation}\label{bounddiam}
\diam(\E_i)\leq K \tv{\E_i}^{1/N}\,.
\end{equation}}

The main result for Riemannian manifolds is the following.

\maintheorem{main2}{Connectedness: the Riemannian case}{
Let $M$ be an $N$-dimensional, ${\rm C}^2$, compact Riemannian manifold, and let $m\in\N$ be given. There exists a constant $\bbeps>0$ such that any minimal cluster $\E$ with total volume $\tv\E<\bbeps$ is connected.}

In fact, what we will prove is slightly stronger. Indeed, the claim is true not only if the manifold is Riemannian, but also if it is ``almost Riemannian''; this means that the infinitesimal norm at each point is not necessarily exactly the Euclidean one, but it is close enough to it. The precise definitions and the proof of this stronger result are contained in the final Section~\ref{afn}. In that section we will also discuss the possibility to further extend Theorem~\mref{main2}. Indeed, one could think that the result might be true for any compact Finsler manifold, which would then make Theorem~\mref{main1} useless, but the claim is false in this generality. Instead, it is reasonable to guess the validity of the claim for all the ``fixed norm manifolds''; see Definition~\ref{fnm} and Remark~\ref{lastremark}.

One of the main tools that we will use in the proof of our main results is Theorem~\ref{smalldiam}, which is an extremely important result of independent interest. The claim of this result is that a sufficiently small connected component of a minimal cluster must have small diameter, in particular the diameter can be uniformly bounded by the $N$-th root of the total volume (see in particular Corollary~\ref{corsmalldiam}). A result saying that ``small volume implies small diameter'' is very useful, but it was never proved before for the case of clusters.

Our results are particularly valuable, as they represent, to the best of our knowledge, the first instance in the literature where the global problem for clusters of small volumes can be rigorously reduced to a local one. It thereby provides a formal proof of a widely expected phenomenon: namely, that isoperimetric clusters of small volumes are small perturbations of the isoperimetric clusters for the corresponding infinitesimal problem in the tangent space. The uniform control of the diameter of an isoperimetric cluster in terms of its volume and the geometric bounds of the ambient manifold paves the way for extending the well-established theory in the case $m = 1 $ to the general case of arbitrary $m\geq 1$.

We believe that the ideas developed in this work are also useful for proving deep results that are genuinely new even in the case $m = 1$, with potential applications in the theory of existence and multiplicity of solutions to various phase transition models arising in realistic physical contexts (see~\cite{andradeponciano2024,andraderesende2024} for comparisons). A first, simple and direct consequence of our results, is that the (multi)-isoperimetric profile in a compact Riemannian manifold is asymptotically close to that of $\R^N$ with the standard Euclidean metric for small total volumes. This result --which is highly expected-- generalizes the classical result by B\`erard and Meyer~\cite{MR690651} relative to the case $m = 1$. For $m = 2$, this was recently established in~\cite[Corollary 3.28]{andraderesende2024} using the full resolution of the double bubble conjecture. In contrast, we can now immediately obtain the completely general case of any $m$ as a simple consequence of the sole Theorem~\ref{smalldiam}, even in the Finsler case.

\begin{theorem}[Asymptotic expansion of the multi-isoperimetric profile for small volumes]\label{AsymptoticTheorem}
Let $M$ be an $N$-dimensional, ${\rm C}^2$, compact Riemannian manifold, and let $m\in\N$ be given. Then
\begin{equation}\label{eq12}
\frac{J_M(v)}{J_{\R^N}(v)} \to 1\qquad \hbox{as}\ |v|_1\to 0\,,
\end{equation}
where $J_M$ and $J_{\R^N}$ denote the multi-isoperimetric profiles of $M$ and of $\R^N$ with the Euclidean metric, and for any vector $v\in(\R^+)^m$ we denote by $|v|_1=v_1 + v_2 + \cdots + v_m$ the $L^1$ norm, which is the ``total volume'' of an $m$-cluster whose volume is $v$. More generally, if $M$ is a compact Finsler manifold, then the following expansion holds,
\begin{equation}\label{eq12bis}
\frac 1{C_2^-} \leq \liminf \frac{J_M(v)}{J_{\R^N}(v)}\leq \limsup \frac{J_M(v)}{J_{\R^N}(v)} \leq C_2^+ \qquad \hbox{as}\ |v|_1\to 0\,,
\end{equation}
where the constants $C_2^-$ and $C_2^+$ are defined in~(\ref{defC_2^-}) and~(\ref{defC_2^+}) respectively. Equation~(\ref{eq12bis}) reduces to~(\ref{eq12}) in the Riemannian case, since in that case $C_2^-=C_2^+=1$.
\end{theorem}

Finally, we emphasize that our results provide the correct generalization of Almgren's Lemma, in the setting where a sequence of $m$-clusters converges to a limit with strictly fewer chambers.

The plan of the paper is the following. In Section~\ref{notdef} we will describe the Finsler manifolds, giving all the relevant definitions and fixing the notation that will be used throughout the paper. Then, Section~\ref{secsmalldiam} will be devoted to show the fundamental Theorem~\ref{smalldiam}. In Section~\ref{secma1} and~\ref{secma2} we will prove Theorem~\mref{main1} and~\mref{main2} respectively. In the short final Section~\ref{secfin}, we will show Theorem~\ref{AsymptoticTheorem} and we will present a couple of final remarks.

\subsection{Notation and basic definitions for Finsler manifolds\label{notdef}}

This section is devoted to briefly present the classical theory of Finsler manifolds, together with our slightly more general setting, and to fix the notation and give the definitions that we will use throughout the paper. For a more detailed study of Finsler geometry, one can refer to the classical book~\cite{MR1747675}, to the recent book~\cite{Ohtabook}, or to the recent papers~\cite{manini2024,ohta2018,ohta2009}, where one can find very clear descriptions.

Roughly speaking, a Finsler manifold is a manifold which, around each point, behaves like $\R^N$ endowed with a convex norm, and this norm depends on the point. The formal definition is the following.
\begin{defin}[smooth Finsler manifold]
Let $M$ be a compact, connected, $N$-dimensional manifold, and let $\mu$ be a smooth, positive measure on $M$ (this means that $\mu$ can be expressed in local coordinates by means of a strictly positive and smooth density). Moreover, let $j_F: TM \to \R^+$ be a \emph{(smooth) Finsler structure}, that is, a positively $1$-homogeneous function, smooth out of the null section, and such that $j_F^2$ is strictly convex on each tangent space $T_p M$. The triple $(M,j_F,\mu)$ is called a \emph{smooth Finsler manifold}.
\end{defin}

In this paper we will use a more general setting, in which we do not require smoothness, but only ${\rm C}^2$ regularity, and only in the spatial variable. More precisely, we give the following.
\begin{defin}[${\rm C}^2$ Finsler manifold]
Let $M$ be a compact, connected, $N$-dimensional manifold, and let $\mu$ be a ${\rm C}^2$, positive measure on $M$ (this means that $\mu$ can be expressed in local coordinates by means of a strictly positive and ${\rm C}^2$ density). Moreover, let $j_F: TM \to \R^+$ be a \emph{${\rm C}^2$ Finsler structure}, that is, a positively $1$-homogeneous function, such that $j_F^2$ is convex on each tangent space $T_p M$, and which is of class ${\rm C}^2$ in the spatial variable out of the null section. The triple $(M,j_F,\mu)$ is called a \emph{${\rm C}^2$ Finsler manifold}.
\end{defin}

It is important to notice the main, big difference between the classical notion of smooth Finsler manifolds and our choice (of course, since our notion is a generalization of the classical one, then all our results are true also in the case of the classical, smooth notion). This difference is not the fact that we only require ${\rm C}^2$ regularity instead of smoothness, this is actually a minor detail without much interest; the big point is that for the Finsler structure we require regularity only in the spatial variable, and only convexity on each tangent space. As a consequence, for every point $p\in M$ the ``unit ball'' in $T_p M$ consisting of the vectors $v\in T_pM$ such that $j_F(v)=1$ is not a strictly convex set, but only a convex one. Even though this introduces some additional technical difficulties, we find this particularly interesting, since it allows to locally observe general Wulff shapes, not necessarily strictly convex ones, and this often gives rise to much more interesting situations and examples.

Another comment has to be made on the measure $\mu$ on $M$. Indeed, in general there are different possible meaningful choices, which of course all reduce to the standard volume measure in the Riemannian case. More precisely, as Ohta says in~\cite{ohta2009}, \emph{``different from Riemannian manifolds, there is a variety of canonical measures on Finsler manifolds, such as the Busemann--Hausdorff measure, the Holmes--Thompson measure and so forth''}.

If $(M,j_F,\mu)$ is a Finsler manifold, one says that the Finsler structure is \emph{reversible} if for any $v\in TM$ one has $j_F(v)=j_F(-v)$. However, in this paper we are only interested in the general irreversible case. Given a Finsler structure on $M$, there is a natural notion of length of curves; that is, the length of a regular curve $\gamma:[0,1]\to M$ is given by
\[
\int_0^1 j_F(\gamma'(t))\,dt\,.
\]
This notion of length automatically defines a (semi)-distance between points, by saying that $d(x,y)$ is the minimal length of curves connecting $x$ to $y$. Notice that this is a semi-distance in the general irreversible case, since the length of a curve can differ upon reverting the orientation. And in turn, having such a distance allows to speak about the space ${\rm Lip}(M)$ of the Lipschitz functions on $M$; indeed, a function $\Phi:M\to\R$ is said to be \emph{Lipschitz with constant $L$} if for every $p,\,q\in M$ one has
\[
\big|\Phi(p)-\Phi(q)\big| \leq L\, \min \big\{ d(p,q),\, d(q,p)\big\}\,. 
\]
Moreover, one can define the \emph{slope} at a point $p$ for a Lipschitz function $\Phi$ as
\[
|\partial \Phi |(p) = \limsup_{q\to p} \frac{\big(\Phi(p)-\Phi(q)\big)^+}{d(p,q)}\,.
\]
We are finally in position to define the notion of perimeter on such a manifold. That is, given a Borel set $E\subseteq M$, we define its perimeter as
\[
P_M(E) = \inf \bigg\{ \liminf \int_M |\partial \Phi_n |(p) \, d\mu(p)\,, \Phi_n \in {\rm Lip}(M),\, \Phi_n \freccia{L^1} \Chi{E}\bigg\}\,.
\]
In this paper, we will find it very useful to calculate volumes and perimeters in local charts. Let us be more precise: let us consider an atlas on $M$ composed of finitely many bijective functions $\varphi_i :U_i \to V_i$, where each $U_i$ is a ball in $\R^N$ and $\{V_i\}$ is an open cover of $M$. Then, the volume and the perimeter of every set $E\comp V_i$ of finite perimeter can be written as
\begin{align*}
|E|_M = \int_F g_i(x)\, d\H^N\,, && P_M(E) = \int_{\partial^* F} h_i\big(x,\nu_F(x)\big)\, d\H^{N-1}(x)\,,
\end{align*}
where $F= \varphi_i^{-1}(E)$, $\partial^* F$ is the reduced boundary of $F$, and the functions $g_i : U_i\to \R^+$ and $h_i: U_i \times \R^N \to \R^+$ are determined by the measure $\mu$ and the structure $j_F$. Of course, the functions $g_i$ and $h_i$ are highly non-canonical, since they depend on the arbitrary choice of the charts. In order to simplify the calculations, throughout the paper we will always consider a \emph{isovolumetric atlas}, that is, an atlas such that the functions $g_i$ in the above formula are identically equal to $1$. In other words, for each chart $\varphi_i:U_i \to V_i$ of an isovolumetric atlas and for each set $E\comp V_i$ of finite perimeter one has
\begin{align}\label{choiceatlas}
|E|_M = |F| = \int_F 1\, d\H^N\,, && P_M(E) = \int_{\partial^* F} h_i\big(x,\nu_F(x)\big)\, d\H^{N-1}(x)\,.
\end{align}
Notice that this choice is not really standard; however, we find it very useful because then the volume of any set $E\comp V_i$ is exactly the same as the Euclidean volume of the reference set $F=\varphi_i^{-1}(E)$. Notice also that the existence of an isovolumetric atlas is obvious; in fact, by any atlas one can obtain an isovolumetric one through a change of variables and possibly a refinement (keep in mind that the finiteness is not a problem since we are interested in compact manifolds).\par

Let us now consider a given point $p \in M$. Let us assume that this point belongs to a set $V_i$, and let us call $z=\varphi_i^{-1}(p)$. Then, the function $h_i(z,\cdot)$ is a positively $1$-homogeneous and convex function from $\R^N$ to $\R^+$, so it corresponds to a norm, that we can call $\|\cdot\|_{i,z}$: the unit ball for this norm is strictly convex for a classical (smooth) Finsler manifold, while it is in general only convex for a ${\rm C}^2$ Finsler manifold. It is clear that, given the Finsler manifold $M$ and the point $p\in M$, the norm $\|\cdot\|_{i,z}$ is not uniquely determined, since it also depends on the choice of the atlas. However, let us consider two different possible atlantes: since they are both isovolumetric, the two different norms corresponding to the same point $p$ coincide up to a change of variable which is given by a linear map of unit-determinant. In other words, the norm $\|\cdot\|_{i,z}$ is actually uniquely determined by the point $p$ up to an arbitrary composition with a unit-determinant linear map. This allows us to give the following definition.

\begin{defin}[Infinitesimal norm]
Let $(M,j_F,\mu)$ be a Finsler manifold, and let $p\in M$. We call \emph{infinitesimal norm at $p$} the norm $\|\cdot\|_{i,z}$ where $z =\varphi_i^{-1}(p)$ and $\{\varphi_i, U_i, V_i\}$ is an arbitrary isovolumetric atlas. The infinitesimal norm is uniquely defined up to composition with a unit-determinant linear map.
\end{defin}

Notice that, having a well-defined norm at any point, we can compare the norms at different points; in particular, we can say if they coincide or not (of course, always up to composition with unit-determinant maps). Then, the following definition is natural.

\begin{defin}[Fixed-norm manifold]\label{fnm}
A Finsler manifold $M$ is said to be a \emph{fixed-norm manifold} if the infinitesimal norms at all the points coincide (up to unit-determinant linear maps). In this case, we call \emph{fixed norm} this unique norm.
\end{defin}

Notice that (thanks to the Nash embedding Theorem), any sufficiently regular Riemannian manifold is a fixed-norm manifold, and in particular the fixed norm is the Euclidean norm.

Let us now again fix our attention to a point $p\in M$ and let us call $K$ the unit ball corresponding to the norm $\|\cdot\|_{i,z}$ for some isovolumetric atlas with $p=\varphi_i(z)$. As said above, for any other isovolumetric atlas the corresponding unit ball at $p$ will be a convex set of the form $L(K)$ for a unit-determinant linear map $L$. And in fact, by an obvious change of variable, for any such map $L$ we can find an isovolumetric atlas such that $L(K)$ is the corresponding ball. The full generality of $L$ is quite annoying, since then there is no upper bound to the perimeter of the possible unit balls at a given point $p$, and for different isovolumetric atlantes the values of $h_i(z,v)$ for unit vectors $v$ might be arbitrarily large or arbitrarily small. This could make it complicate to give stable estimates. To solve this issue, we will select unit balls which are ``balanced'', in the following sense.
\begin{defin}[Balanced convex set and balanced infinitesimal norm]
Let $K\subseteq \R^N$ be a compact, convex set whose interior contains the origin, and let $0<r\leq R$ be the inradius and the circumradius of $K$, that is, the radii of the largest ball contained in $K$ and of the smallest ball containing $K$. We say that $K$ is \emph{balanced} if for any unit-determinant linear map $L$, calling $0<r'\leq R'$ the inradius and the circumradius of $L(K)$, one has $\frac R r\leq \frac {R'}{r'}$. Let also $(M,j_F,\mu)$ be a Finsler manifold, and let $p\in M$. We call \emph{balanced infinitesimal norm at $p$} any norm $\|\cdot\|_{i,z}$ if $z=\varphi_i^{-1}(p)$ and $\{\varphi_i, U_i, V_i\}$ is an arbitrary isovolumetric atlas with the property that the unit ball corresponding to $\|\cdot\|_{i,z}$ is balanced. We will use the notation $\|\cdot\|_p$ for any such norm, and $P_p$ for the corresponding \emph{balanced infinitesimal perimeter}.
\end{defin}

It is immediate to observe that, for any compact, convex set $K$ whose interior contains the origin, there is at least one unit-determinant linear map $L$ such that $L(K)$ is balanced --in particular, by the well-known John's Lemma, for any such balanced set one has $\frac R r\leq 2N$. As a consequence, for any point $p\in M$ at least one balanced infinitesimal norm $\|\cdot\|_p$ exists; they are in general not unique (any rotation of a balanced infinitesimal norm at $p$ is still a balanced infinitesimal norm at $p$), but they are not as many as the (not necessarily balanced) infinitesimal norms. From now on, we will restrict our attention to the following particular atlantes.

\begin{defin}
Let $(M,j_F,\mu)$ be a Finsler manifold, and let $\{\varphi_i,\, U_i,\, V_i\}$ be a (finite) isovolumetric atlas. We say that the atlas is \emph{balanced} if for any index $i$ there exists a point $x_i\in U_i$ such that the convex set $\{v \in \R^N,\, h_i(x_i,v)\leq 1\}$ is balanced.
\end{defin}

Notice that the existence of balanced atlantes is obvious, since they can be obtained starting from any isovolumetric atlas through a change of variables. Throughout the paper, we will only consider balanced atlantes. We can now present the last definitions that we are going to use in this paper; they might seem quite technical, but their meaning is pretty obvious. The first one is the oscillation of an atlas.
 
\begin{defin}[oscillation of an atlas]\label{defosc}
Let $\big\{ \varphi_i,\, U_i,\, V_i\big\}$ be a balanced atlas, and let $\rho>0$ be a small number. We say that the atlas has \emph{oscillation bounded by $\rho$} if for any index $i$, any two points $x,\,y$ in $U_i$, and any unit vector $\nu\in\S^{N-1}$ one has
\[
1-\rho \leq \frac{h_i(x,\nu)}{h_i(y,\nu)} \leq 1+\rho\,. 
\]
\end{defin}
For any given $\rho\ll 1$, it is clearly possible to take a balanced atlas with oscillation bounded by $\rho$, with number of charts exploding when $\rho\to 0$. It will be important to use atlantes with very small oscillation, because then the functions $h_i(x,\nu)$ will be quantitatively close to functions depending only on $\nu$, that is, the situation in any chart will be very similar to that of $\R^N$ with a fixed norm.\par

Let us now pass to define the $m$-clusters and their volume and perimeter. Given $m\in\N$, an \emph{$m$-cluster} is a collection of $m$ mutually disjoint sets $\E=(E_1,\, E_2,\, \dots\,,\, E_m)$ in $M$. The \emph{volume} of the cluster $\E$ is the vector $|\E|_M = (|E_1|_M,\, |E_2|_M,\, \dots\,,\, |E_m|_M)\in (\R^+)^m$, and the \emph{total volume} of the cluster is $\tv\E_M=\big| \cup_{j=1}^m E_j\big|_M$, so the $L^1$ norm of the vector $|\E|_M\in\R^m$. Let us now define the \emph{perimeter} of $\E$. In the special Euclidean case, the perimeter of $\E$ is simply $\H^{N-1}(\partial^*\E)$, where we write for brevity
\[
\partial^* \E := \bigcup\nolimits_{j=1}^m \partial^* E_j\,.
\]

Notice that the perimeter of a cluster is smaller than the sum of the perimeters of the sets $E_j$, because the ``common boundary'' is counted just once. Therefore, in order to minimize the perimeter of clusters with given volume, it is easy to guess that it is convenient to have the different sets attached to each other in order to have a sufficiently large ``common boundary''. In the general non-Euclidean case, the standard definition of perimeter of the cluster is
\begin{equation}\label{perclu}
P_M(\E) = \frac 12\, \bigg(P_M\Big( \bigcup\nolimits_{j=1}^m E_j\Big) + \sum\nolimits_{j=1}^m P_M(E_j)\bigg)\,.
\end{equation}
This definition might seem strange at first sight, but it is immediate to observe that this actually reduces to $\H^{N-1}(\partial^*\E)$ in the Euclidean case, and that this is exactly what one has in mind as ``perimeter'' for a cluster in a manifold. It is useful to make another observation, namely, if the manifold is reversible then the perimeter of a cluster can be easily written as an integral. More precisely, if the whole cluster $\E$ is entirely contained in the chart $V_i$, and we call $\F=\varphi_i^{-1}(\E)$, then we can write
\[
P_M(\E) = \int_{\partial^*\F} h_i(x,\nu_\F(x)) d\H^{N-1}(x)\,.
\]
Observe that the normal vector $\nu_\F(x)$ is only defined up to the sign, and then the above formula makes sense only for a reversible manifold where $h_i(x,\nu)=h_i(x,-\nu)$. As a consequence, since we are not going to restrict ourselves to the case of reversible Finsler manifolds, we will only make use of the general definition~(\ref{perclu}).\par

As usual, we will say that $\E$ is a \emph{minimal cluster} if it minimizes the perimeter among all clusters with the same volume (keep in mind that the ``volume'' is not a number, but an $m$-dimensional vector). By standard lower semicontinuity and compactness, there exist minimal clusters for any given volume, as soon as the total volume is less than the volume of $M$ itself.\par

Now, we give another definition.

\begin{defin}[maximal diameter]\label{defmadi}
The \emph{maximal diameter $d$} for the manifold $M$ is the maximum of the diameters of all the $m$-clusters in $\R^N$ which have total volume equal to $1$ and which are minimal clusters with respect to some balanced infinitesimal perimeter $P_p$ for some $p\in M$.
\end{defin}
Notice that the maximal diameter $d$ is well defined by the compactness of $M$ and of the subset of $(\R^+)^m$ consisting of the vectors with $L^1$ norm equal to $1$; notice also that having considered balanced infinitesimal perimeters was essential; if we used instead the notion of (not necessarily balanced) infinitesimal perimeters, then the maximal diameter would be $+\infty$.

By definition, whenever an $m$-cluster is minimal for some of the perimeters $P_p$, its diameter must be bounded by $d$ (times the $N$-th root of the total volume, if different from $1$). By approximation, we can also deduce that if some cluster is \emph{almost} minimal for some perimeter, then a \emph{large} portion of it has diameter at most $2d$ (we could do the same with $cd$ for any $c>1$, but this would only make the notation more complicated). The corresponding definition is the following.

\begin{defin}[$\sigma$-minimal deviation]\label{defsmd}
Let $M$ be a manifold and $d$ its maximal diameter. For every small constant $\sigma$, we say that the \emph{$\sigma$-minimal deviation} is the largest constant $\eta=\eta(\sigma)$ such that, if $\F\subseteq\R^N$ is any cluster with total volume $1$ and such that for some balanced infinitesimal perimeter $P_p$ one has
\[
P_p(\F) \leq (1+\eta) \min \Big\{ P_p(\widetilde\F),\, |\widetilde\F| = |\F|\Big\}\,,
\]
then there exists a ball $B$ of diameter $2d$ such that $\tv{\F\cap B}\geq 1-\sigma$.
\end{defin}
In words, a cluster which is optimal up to a percentage $\eta(\sigma)$ for some perimeter $P_p$ must be contained in some ball of radius $d$ except for a small portion with total volume less than $\sigma$ (once again, the well-posedness of this definition is trivial by compactness).\par

We conclude this introduction by defining three constants (as usual, well-defined by compactness) which will be used in the sequel, and which depend only on the manifold and its structure, and \emph{not} on the choice of a particular atlas (in particular, they may be large, but they do not explode by taking atlantes with arbitrarily small oscillation). In these definition, by $P_p$ (resp., $\|\cdot\|_p$) we always mean a balanced infinitesimal perimeter (resp., norm) for some $p\in M$.
\begin{gather}
C_1= \Big(\min \Big\{ \H^{N-1}(\partial^* \F),\, \tv{\F}=1,\, \hbox{$\F$ is a minimal cluster for some $P_p$}\Big\}\Big)^{-1}\,,\label{defC_1}\\
C^+_2 = \max \Big\{ \|\nu\|_p,\, p\in M,\, \nu\in\S^{N-1} \Big\}\,, \label{defC_2^+}\\
C^-_2 = \Big(\min \Big\{ \|\nu\|_p,\, p\in M,\, \nu\in\S^{N-1} \Big\}\Big)^{-1} \,.\label{defC_2^-}
\end{gather}
Notice the power $-1$ in the definition of $C_1$ and $C_2^-$; we have made this choice just for simplicity, so that for a ``bad'' manifold all the three constants might become very large, instead of someone very large and some other very small. Finally, notice that if the Finsler manifold is actually a Riemannian manifold, then the unique balanced infinitesimal norm is the Euclidean norm, and then in particular $C^-_2=C^+_2=1$.


\section{The small connected components\label{secsmalldiam}}

This section is devoted to show that, for a minimal cluster of sufficiently small total volume, every connected component is contained in a single chart of a suitable atlas, and its diameter is very small. More precisely, the main result of this section is the following.
\begin{theorem}\label{smalldiam}
Let $M$ be a $N$-dimensional, ${\rm C}^2$, compact, Finsler manifold, and $m\in\N$ be given. For any small $\sigma<4^{-N}$, and for any balanced atlas with oscillation smaller than $\sigma$, there exists a constant $\bar\eps_1>0$ such that the following holds. Let $\E\subseteq M$ be a minimal $m$-cluster, and let $\widetilde\E$ be a connected component of $\E$ with total volume less than $\bar\eps_1$. Then, there exists an index $i$ such that $\widetilde\E$ is entirely contained in the chart $V_i$, and the cluster $\F= \varphi_i^{-1}(\widetilde \E)$ satisfies $\diam(\F)\leq 4C_4 \tv\F^{1/N}=4C_4 \tv{\widetilde \E}_M^{1/N}$, where $C_4$ is defined in~(\ref{defC4}).
\end{theorem}

The meaning of this result is that small connected components of a minimal cluster have small diameter. The claim makes use of a specific balanced atlas only for later use, but an immediate corollary of the result is the following general fact, of its own interest.

\begin{corol}\label{corsmalldiam}
Let $M$ be a $N$-dimensional, ${\rm C}^2$, compact Finsler manifold, and $m\in\N$ be given. There exist a geometric constant $\overline C$ and a small number $\bar\eps_1>0$, only depending on the geometry of $M$, such that the following holds. Let $\E\subseteq M$ be a minimal $m$-cluster, and let $\widetilde\E$ be a connected component of $\E$ with total volume less than $\bar\eps_1$. Then the diameter of $\widetilde \E$ is smaller than $\overline C \tv\E^{1/N}$.
\end{corol}

Notice that the claim of the corollary is just a part of the claim of the more general Theorem~\ref{smalldiam}; one only needs to notice that the ratio between the ``true'' distance in the Finsler manifold and the ``artificial'' distance measured in the charts is bounded only in terms of the geometry of the manifold, by definition and by the compactness of $M$.

In the above claims, by ``connected component of a cluster $\E$'' we mean any connected component of the union $\cup_{j=1}^m E_j$, so in particular any such component is an $m$-cluster itself, possibly with some empty chamber. Observe that the result does not require that the total volume of the whole minimal cluster is small, but only that so is the total volume of the connected component $\E$.\par

In order to state the next technical results, we need another simple piece of notation. Given a ball $B=B(c,r)\subseteq\R^N$, with radius $r$ and center $c\in\R^N$, for every $A\subseteq [0,r]$ we write
\[
B_A := \Big\{ y \in B:\, |y-c|\in A \Big\} = \bigcup \big\{ \partial B(c,s),\, s\in A\big\}\,,
\]
so that in particular $B=B_{[0,r]}$. The first result of this section is then the following.
\begin{lemma}\label{lemma1}
Let $\sigma>0$ be given, let $\eta$ be the associated $\sigma$-minimal deviation, and let $\{\varphi_i,\, U_i,\, V_i\}$ be a balanced atlas with oscillation smaller than $\rho\leq\eta/4$. Let the ball $B(c,r)$ be compactly contained in $U_i$. Let also $\F$ be a connected component of $\varphi_i^{-1}(\E\cap V_i)$, where $\E$ is some minimal cluster, and let $I=[a,b]\subseteq [0,r]$ be an interval such that, calling $\eps_I =\tv{\F\cap B_I}$,
\begin{equation}\label{assumptlemma1}
\H^{N-1}\Big(\big(\partial B(c,a)\cup \partial B(c,b)\big)\cap \F\Big) \leq \frac 1{C_3}\, \eps_I^{\frac{N-1}N}\,,
\end{equation}
with
\begin{equation}\label{defC3}
C_3 = \frac{2C_1 C_2^- C_2^+}\rho\,.
\end{equation}
Then, there exists some interval $J\subseteq I$ such that
\begin{align}\label{thesislemma1}
|J| \leq 2d \eps_I^{1/N} \,, && \tv{\F \cap B_J} \geq \big(1-\sigma) \eps_I\,.
\end{align}
\end{lemma}
\begin{proof}
Let us call for brevity $\F_I=\F\cap B_I$, and let $\F'$ be a cluster with volume $|\F'|=|\F_I|\in\R^m$ which is minimal for the balanced infinitesimal perimeter $P_p$, where $p\in V_i$ is a point such that the convex set $\{v\in\R^N,\, h_i(\varphi_i^{-1}(p),v)\leq 1\}$ is balanced. By definition of maximal diameter, we have that
\[
\diam(\F') \leq d \tv{\F'}^{1/N} = d \tv{\F_I}^{1/N}=d\,\eps_I^{1/N}\,.
\]
If $|I|\leq d\, \eps_I^{1/N}$, then we can simply set $J=I$ and~(\ref{thesislemma1}) is clearly true. Hence, we can assume without loss of generality that $|I|> d\,\eps_I^{1/N}$. In this case, up to a translation we have that $\F'$ is compactly contained in the annulus $B_I$. Therefore, we can define a competitor cluster $\E'$ simply removing $\varphi_i(\F_I)$ from $\E$ and replacing it with $\varphi_i(\F')$. By construction, $\tv{\E'} = \tv\E$, so the cluster $\E'$ is actually a competitor for $\E$. To evaluate the difference between the perimeters of $\E$ and $\E'$, we call $\Gamma=\big(\partial B(c,a) \cup \partial B(c,b)\big)\cap\F$, and for every $x\in\Gamma$ we call $\nu(x)$ the direction of the vector $x-c$, so the normal vector (up to the sign) to the annulus $B_I$ at $x$. Thus, keeping in mind~(\ref{choiceatlas}), by construction we have
\[
P_M(\E')-P_M(\E) = P_M(\varphi_i(\F')) - P_M(\varphi_i(\F_I)) + \int_\Gamma h_i(x,\nu(x))+h_i(x,-\nu(x)) \,d\H^{N-1}(x)\,.
\]
Since the oscillation of the atlas is less than $\rho$, we have
\begin{align*}
P_M(\varphi_i(\F_I))\geq (1-\rho) P_p(\F_I)\,, && P_M(\varphi_i(\F'))\leq (1+\rho) P_p(\F')\,.
\end{align*}
Moreover, for any $x\in\Gamma\subseteq U_i$ and any $\nu\in\S^{N-1}$, by~(\ref{defC_2^+}) we have
\[
h_i(x,\nu)  \leq (1+\rho) h_i(c,\nu) = (1+\rho) \|\nu\|_p \leq (1+\rho) C_2^+\,.
\]
Keeping in mind that $\E$ is a minimal cluster, thus $P_M(\E)\leq P_M(\E')$, we deduce
\[\begin{split}
(1-\rho) P_p(\F_I) &\leq P_M(\varphi_i(\F_I)) \leq P_M(\varphi_i(\F')) + \int_\Gamma h_i(x,\nu(x))+h_i(x,-\nu(x)) \,d\H^{N-1}(x)\\
&\leq (1+\rho) P_p(\F') + 2(1+\rho) C_2^+ \H^{N-1}(\Gamma)
\leq (1+\rho) \bigg( P_p(\F') + \frac{\rho \,\eps_I^{\frac{N-1}N}}{C_1C_2^-} \bigg)\,,
\end{split}\]
where we have used~(\ref{assumptlemma1}) and~(\ref{defC3}). Now, remember that $\F'$ is a minimal cluster for the perimeter $P_p$, and it has total volume $\eps_I$; hence, by~(\ref{defC_2^-}) and~(\ref{defC_1}), we have that
\[
P_p(\F') \geq \frac 1 {C_2^-}\, \H^{N-1}(\partial^* \F') \geq \frac{\eps_I^{\frac{N-1}N}}{C_1 C_2^-}\,.
\]
Inserting this bound in the above inequality, and keeping in mind that $\rho$ is small, we get
\[
P_p(\F_I) \leq \frac{(1+\rho)^2}{1-\rho} \, P_p(\F')\leq (1+4\rho) P_p(\F')\leq (1+\eta) P_p(\F')  \,.
\]
Finally, recalling that $|\F_I|=|\F'|$, that $\F'$ is a minimal cluster for the perimeter $P_p$, and that $\eta$ is the $\sigma$-minimal deviation, by Definition~\ref{defsmd} we obtain the existence of a ball $D$ with diameter $2d\eps_I^{1/N}$ and such that $\tv{\F_I\cap D}\geq (1-\sigma) \eps_I$. We can finally call $J\subseteq I$ the smallest interval such that $D \cap B_I \subseteq B_J$, and notice that this interval satisfies the requirements of~(\ref{thesislemma1}).
\end{proof}

Our second result says that the thesis of Lemma~\ref{lemma1} remains true also removing the assumtion~(\ref{assumptlemma1}), up to replace the factor $2$ in~(\ref{thesislemma1}) with a larger one, and to allow $J$ to be not an interval, but the union of three intervals.
\begin{lemma}\label{lemma2}
Let $\sigma,\, \eta,\, \{\varphi_i,\, U_i,\, V_i\},\, \rho,\, B, \, \E,\, \F$ be as in Lemma~\ref{lemma1}, with $\rho=\eta/4$, let again $I\subseteq [0,r]$ be an interval, and call $\eps_I=\tv{\F\cap B_I}$. Then, there exist three disjoint intervals $J' = (a,a^+)$, $J'' = (a',b')$ and $J'''=(b^-,b)$ such that $J=J'\cup J'' \cup J'''$ satisfies
\begin{align}\label{thesislemma2}
|J| \leq C_4 \eps_I^{1/N} \,, && \tv{\F \cap B_J} \geq \big(1-\sigma) \eps_I\,,
\end{align}
where
\begin{equation}\label{defC4}
C_4 = \frac{8 C_3}{1-2^{-1/N}}+2d\,.
\end{equation}
\end{lemma}
\begin{proof}
We will perform an argument by recursion. We start by noticing that, if for every $t\in (a, (a+b)/2)$ one has
\begin{equation}\label{checkthis}
\H^{N-1}\big( \F\cap \partial B(c,t)\big) \geq \frac 1{4C_3}\,\eps_I^{\frac{N-1}N}\,,
\end{equation}
then by Fubini
\[
\eps_I = \tv{\F \cap B_I} \geq \int_a^{\frac{a+b}2} \H^{N-1} \big( \F\cap \partial B(c,t)\big)\,dt\geq \frac 1{8C_3}\, |I|\, \eps_I^{\frac{N-1}N}\,,
\]
which gives
\[
|I|\leq  8 C_3  \eps_I^{1/N}\leq C_4\eps_I^{1/N}\,.
\]
Then, the thesis is already obtained with $J=I$ (so $J'=J'''=\emptyset$ and $J''=I$). In the very same way, the thesis immediately follows if~(\ref{checkthis}) is true for every $t \in ((a+b)/2,b)$.

Otherwise, we let $a_1$ be an ``almost smallest'' element of $(a,(a+b)/2)$ such that~(\ref{checkthis}) fails. This means that~(\ref{checkthis}) fails at $a_1$, but it is true for all $t\in (a,(a+a_1)/2)$. Analogously, we let $b_1$ be an ``almost largest'' element of $((a+b)/2,b)$ such that~(\ref{checkthis}) fails, that is, (\ref{checkthis}) is true for all $t\in ((b_1+b)/2,b)$, but false at $b_1$. The same calculation as before via Fubini gives
\[
\eps_I \geq \int_{(a,(a+a_1)/2)\cup ((b_1+b)/2,b)} \H^{N-1} \big( \F\cap \partial B(c,t)\big)\,dt\,,
\]
which gives
\begin{equation}\label{length1}
(a_1-a) + (b-b_1) \leq  8C_3 \eps_I^{1/N}\,.
\end{equation}
Let us now call $I_1=[a_1,b_1]$, and let $\eps_1 =\tv{\F\cap B_{I_1}}$. Since we know that~(\ref{checkthis}) fails at $a_1$ and $b_1$, we have
\begin{equation}\label{failure}
\H^{N-1} \big( (\partial B(c,a_1)\cup \partial B(c,b_1))\cap \F\big) \leq \frac  1{2C_3}\, \eps_I^{\frac{N-1}N}\,.
\end{equation}
We can now distinguish two possibilities. First, assume that $\eps_1\geq \eps_I/2$, so that the above estimate implies
\[
\H^{N-1} \big( (\partial B(c,a_1)\cup \partial B(c,b_1))\cap \F\big) \leq \frac {2^{-1/N}} {C_3}\, \eps_1^{\frac{N-1}N}\leq \frac 1{C_3}\, \eps_1^{\frac{N-1}N} \,.
\]
Thus, we can apply Lemma~\ref{lemma1} with the interval $I_1=(a_1,b_1)$ in place of $I$, since the above estimate corresponds to~(\ref{assumptlemma1}). Lemma~\ref{lemma1} provides then us with an interval $J''\subseteq I_1$ satisfying~(\ref{thesislemma1}), which reads as
\begin{align*}
|J''| \leq 2 d \eps_1^{1/N} \leq 2d \eps_I^{1/N}\,, && \tv{\F\cap B_{J''}} \geq (1-\sigma) \eps_1\,.
\end{align*}
The thesis is then obtained with $J'=(a,a_1)$ and $J'''=(b_1,b)$. Indeed, by the above estimate and~(\ref{length1}), we have
\[
|J| = |J'| + |J''| + |J'''| \leq (8C_3 +2d )\eps_I^{1/N} \leq C_4 \eps_I^{1/N}\,,
\]
and
\[
\tv{\F\cap B_J} = \eps_I - \eps_1 + \tv{\F\cap B_{J''}} \geq \eps_I - \sigma \eps_1 \geq (1-\sigma) \eps_I\,,
\]
so that~(\ref{thesislemma2}) is true. Hence, we have obtained the proof under the assumption that $\eps_1\geq \eps_I/2$.

Let us now assume that $\eps_1< \eps_I/2$. In this case, we start again as in the beginning of the proof, using $I_1$ and $\eps_1$ in place of $I$ and $\eps_i$. In place of~(\ref{checkthis}), we are now interested in the validity of the inequality
\begin{equation}\tag{\ref{checkthis}'}\label{checkthis'}
\H^{N-1}\big( \F\cap \partial B(c,t)\big) \geq \frac 1{4C_3}\,\eps_1^{\frac{N-1}N}
\end{equation}
for $t\in I_1$. If this inequality is true for all $t\in (a_1, (a_1+b_1)/2)$, or for all $t\in ((a_1+b_1)/2,b_1)$, the same argument as before by Fubini implies
\[
|I_1| \leq 8C_3 \eps_1^{1/N} \,,
\]
which using~(\ref{length1}) and keeping in mind that $\eps_1<\eps_I/2$ implies
\[
|I| \leq |I_1| + 8C_3 \eps_I^{1/N} \leq 8 C_3 \Big(  1+2^{-1/N}\Big)\eps_I^{1/N}\,.
\]
Therefore, also in this case the thesis is completed simply with $J=I$. Otherwise, arguing as in the previous step we call $a_2$ (resp., $b_2$) an ``almost smallest'' (resp., ``almost largest'') element of $(a_1,(a_1+b_1)/2)$ (resp., $((a_1+b_1)/2,b_1)$) such that~(\ref{checkthis'}) fails, and we denote $I_2=[a_2,b_2]$ and $\eps_2=\tv{\F\cap I_2}$. The same arguments as before give now, in place of~(\ref{length1}) and~(\ref{failure}), the estimates
\begin{gather}
(a_2-a_1) + (b_1-b_2) \leq 8 C_3 \eps_1^{1/N}\leq 8\cdot 2^{-\frac 1N} C_3\, \eps_I^{1/N}\,, \tag{\ref{length1}'}\label{length1'}\\
\H^{N-1} \big( (\partial B(c,a_2)\cup \partial B(c,b_2))\cap \F\big) \leq \frac 1 {2C_3}\, \eps_1^{\frac{N-1}N}\,.\tag{\ref{failure}'}\label{failure'}
\end{gather}
As before, if $\eps_2\geq \eps_1/2$ then~(\ref{failure'}) implies~(\ref{assumptlemma1}) with the interval $I_2$ and $\eps_2$ in place of $I$ and $\eps_I$, so we can apply Lemma~\ref{lemma1} getting an interval $J''\subseteq I_2$ satisfying~(\ref{thesislemma1}), that is,
\begin{align*}
|J''| \leq 2 d \eps_2^{1/N} \leq 2d \eps_I^{1/N}\,, && \tv{\F\cap B_{J''}} \geq (1-\sigma) \eps_2\,.
\end{align*}
We obtain then the thesis by setting $J'=(a,a_2)$ and $J'''=(b_2,b)$, since~(\ref{length1}) and~(\ref{length1'}) give
\[
|J| \leq \Big(8C_3  \big( 1 + 2^{-1/N}\big)+ 2d\Big)\eps_I^{1/N}\leq C_4 \eps_I^{1/N}\,,
\]
while by construction and by~(\ref{thesislemma1}) we also have
\[
\tv{\F\cap B_J} = \eps_I - \eps_2 + \tv{\F\cap B_{J''}} \geq \eps_I - \sigma \eps_2 \geq (1-\sigma) \eps_I\,,
\]
so~(\ref{thesislemma2}) is proved. Thus, we are done if $\eps_2\geq \eps_1/2$.

If $\eps_2<\eps_1/2 <\eps_1/4$, we continue recursively in the obvious way. Let us check that this works. After $n\geq 2$ steps, if the analogous of~(\ref{checkthis}) holds in half of the interval $I_n=[a_n,b_n]$, then by Fubini we get $|I_n|\leq 8C_3 \eps_n^{1/N}$, which by~(\ref{length1}) and the anologous estimates, and keeping in mind that $\bal\eps_n\leq \frac{\eps_{n-1}}2\leq \cdots \leq \frac{\eps_I}{2^n}\eal$, implies
\[\begin{split}
|I| &\leq 8C_3 \big( \eps_I^{1/N} + \eps_1^{1/N} + \cdots + \eps_n^{1/N}\big) \leq 8 C_3 \eps_I^{1/N} \bigg(1 + 2^{-1/N}+ 2^{-2/N} + \cdots + 2^{-n/N}\bigg)\\
&\leq \frac{8 C_3}{1-2^{-1/N}} \,\eps_I^{1/N}\leq C_4 \eps_I^{1/N}\,,
\end{split}\]
so the claim is simply true with $J=I$. Otherwise, we can call $I_{n+1}=[a_{n+1},b_{n+1}]$ where $a_{n+1}$ and $b_{n+1}$ are an almost smallest and an almost largest element of $(a_n, (a_n+b_n)/2)$ and of $((a_n+b_n)/2,b_n)$ such that the analogous of~(\ref{checkthis}) does not hold. As before, we have the analogous of~(\ref{length1}) and~(\ref{failure}), so that we have to check whether $\eps_{n+1}\geq \eps_n/2$. If so, then we can apply Lemma~\ref{lemma1} with the interval $I_{n+1}$ and $\eps_{n+1}=\tv{\F\cap B_{I_{n+1}}}$, getting and interval $J''$ satisfying~(\ref{thesislemma1}), thus $|J''| \leq 2d \eps_I^{1/N}$ and $\tv{\F\cap J''}\geq (1-\sigma)\eps_{n+1}$, and then the claim is obtained with $J'=(a,a_{n+1})$ and $J'''=(b_{n+1},b)$, since
\[
|J| \leq \bigg(8C_3 \Big(1 + 2^{-1/N} + \cdots + 2^{-n/N}\Big) +2d\bigg) \eps_I^{1/N}
\leq \bigg( \frac{8C_3}{1-2^{-1/N}} + 2d\bigg)\, \eps_I^{1/N} = C_4 \eps_I^{1/N}
\]
and
\[
\tv{\F\cap B_J} = \eps_I - \eps_{n+1} + \tv{\F\cap J''} \geq (1-\sigma) \eps_I\,.
\]
If not, we go on with the recursion. Thus, either the recursion stops at some point, and then we have obtained the thesis with a suitable $J$, or the recursion never stops, and then we can call $a_\infty$ and $b_\infty$ the limits of the increasing sequence $\{a_n\}$ and of the decreasing sequence $\{b_n\}$. By construction,
\[
\tv{\F\cap B_{(a_\infty,b_\infty)}} = \lim_{n\to\infty} \tv{\F\cap B_{(a_n,b_n)}} = \lim_{n\to\infty} \eps_n \leq \lim_{n\to\infty} 2^{-n} \eps_I = 0\,,
\]
while
\[
(a_\infty - a ) + (b-b_\infty)\leq \frac{8C_3}{1-2^{-1/N}}\, \eps_I^{1/N} \leq C_4 \eps_I^{1/N}\,.
\]
Then, in this last case the claim is obtained with $J'=(a,a_\infty)$, $J''=\emptyset$, and $J'''=(b_\infty,b)$.
\end{proof}

Thanks to the above lemma, we are now in position to show the main result of this section.

\begin{proof}[Proof of Theorem~\ref{smalldiam}]
Let $\sigma<4^{-N}$ be fixed, let $\eta>0$ be the $\sigma$-minimal deviation, and define $\rho=\min\{ \eta/4,\, \sigma\}$, which is a constant only depending on $M$ and on $\sigma$. Let us take a balanced atlas with oscillation smaller than $\rho$, so in particular smaller than $\sigma$. By compactness, there exists some $\bar r>0$, which depends on the atlas, on $M$ and on $\sigma$, such that for each $p\in M$ there exists some index $i$ such that $p\in V_i$, and the ball $B(\varphi_i^{-1}(p),\bar r)$ is compactly contained in $U_i$. We define then
\begin{equation}\label{defbarr}
\bar\eps_1 := \min \bigg\{ \frac 12\, \omega_N ,\, \frac 1{(2C_4)^N}\bigg\} \ \bar r^N\,.
\end{equation}
In particular, the volume of every ball of radius $\bar r$ entirely contained in some $U_i$ is at least $2\bar\eps_1$.\par

Let now $\E$ be any minimal $m$-cluster, and let $\widetilde\E$ be any of its connected components with total volume less than $\bar\eps_1$. Fix an arbitrary point $p\in \partial^*\widetilde\E$, and let $i$ be an index such that $p\in V_i$ and the ball centered at $c=\varphi_i^{-1}(p)$ and with radius $\bar r$ is compactly contained in $U_i$. Let us call $\F$ the connected component of $\varphi_i^{-1}(\widetilde\E\cap V_i)$ containing $c$, and let $\eps=\tv\F$, which is smaller than $\bar\eps_1$ since $\varphi_i(\F)$ is contained in $\widetilde\E$. We also call $B=B(c,\bar r)\comp U_i$, $I_0 = [0,\bar r]$, and $\eps_0 = \tv{\F\cap B_{I_0}}\leq \eps\leq \bar\eps_1$. We can apply Lemma~\ref{lemma2} to the interval $I_0$, getting a set $J_1$ satisfying
\begin{align*}
|J_1| \leq C_4 \eps_0^{1/N} \,, && \eps_1=\tv{\F \cap B_{I_1}} \leq \sigma\eps_0\,,
\end{align*}
where we have called $I_1=I_0\setminus J_1$. Since $J_1=J_1'\cup J_1''\cup J_1'''$, and $J_1',\, J_1''$ and $J_1'''$ are three intervals, with $J_1'$ starting at the first endpoint of $I_0$ and $J_1'''$ ending at the second endpoint of $I_0$, we deduce that $I_1$ is the union of two disjoint intervals, that we call $I_1^1$ and $I_1^2$. Accordingly, we call $\eps_1=\eps_1^1+\eps_1^2$, where $\eps_1^l=\tv{\F\cap B_{I_1^l}}$ for $l\in \{1,\,2\}$. We can now apply Lemma~\ref{lemma2} separately to the intervals $I_1^1$ and $I_1^2$, obtaining two sets $J_2^1$ and $J_2^2$ which satisfy
\begin{align*}
|J_2^l| \leq C_4 (\eps_1^l)^{1/N} \,, && \eps_2^l=\tv{\F \cap B_{I_2^l}} \leq \sigma\eps_1^l\,,
\end{align*}
where $l\in\{1,\,2\}$ and we call $I_2^l=I_1^l\setminus J_2^l$. Setting $J_2=J_2^1\cup J_2^2$, and then also $I_2=I_1\setminus J_2=I_2^1\cup I_2^2$ and $\eps_2=\eps_2^1+\eps_2^2=\tv{\F\cap B_{I_2}}$, the last estimates become
\begin{align*}
|J_2| \leq C_4 \Big((\eps_1^1)^{1/N}+(\eps_1^2)^{1/N}\Big)\leq 2 C_4 \eps_1^{1/N} \,, && \eps_2 \leq \sigma\eps_1\leq \sigma^2 \eps_0\,.
\end{align*}
Going on with the obvious recursion, we find a sequence of disjoint sets $J_n$ such that each set $I_n=I_0\setminus (J_1\cup J_2 \cdots \cup J_n)$ is made by $2^n$ disjoint intervals, and so that $\eps_n = \tv{\F\cap B_{I_n}}\leq \sigma^n \eps_0$, while
\[
|J_{n+1}| \leq 2^n\, C_4  \eps_n^{1/N} \leq \Big(2\sigma^{1/N}\Big)^n C_4 \eps_0^{1/N}\,.
\]
We can then call $J_\infty= \cup_{n\in\N} J_n$, and $I_\infty=I_0\setminus J_\infty$. By the above estimates,
\[
|J_\infty| \leq C_4 \eps_0^{1/N} \sum\nolimits_{n\in\N} \Big(2\sigma^{1/N}\Big)^n = \frac{C_4}{1-2\sigma^{1/N}}\, \eps_0^{1/N}
< 2C_4 \tv\F^{1/N}\leq \bar r\,,
\]
where the last inequalities are due to the fact that $\sigma<4^{-N}$ and by~(\ref{defbarr}). On the other hand, keeping in mind that the sets $I_n$ are monotone decreasing (that is, for every $n$ one has that $I_{n+1}\subseteq I_n$), we deduce that $\tv{\F\cap B_{I_\infty}}=0$. Keeping in mind that $\F$ is connected, we deduce that $J_\infty$ is essentially a segment. More precisely, there exists an interval $[a,b]\subseteq I_0$ such that $\F\cap B_{I_0}$ is concentrated on $B_{[a.b]}$, and $b-a \leq |J_\infty|<\bar r$. Since $c\in \partial^*\F$, it is necessarily $c\in B_{[a,b]}$, and this means that $a=0$. Therefore, we also have $b < \bar r$, and this means that the whole connected component $\F$ is compactly contained in $B(c,\bar r)\comp U_i$, which implies that $\widetilde \E \comp V_i$. We conclude the proof observing that by construction
\[
\diam(\F) \leq 2b\leq 2|J_\infty| \leq 4 C_4 \tv\F^{1/N} = 4 C_4 \tv{\widetilde\E}_M^{1/N}\,.
\]

\end{proof}

\section{The general Finsler case: proof of Theorem~\mref{main1}\label{secma1}}

This section is devoted to prove Theorem~\mref{main1}, that is, that sufficiently small minimal $m$-clusters in $M$ can have at most $m$ connected components. We start with the preliminary observation that every minimal cluster (not necessarily with small voume) contains a finite number of connected components.

\begin{lemma}\label{fmcc}
Let $\E$ be a minimal cluster in $M$. Then, $\E$ is done by finitely many connected components.
\end{lemma}
\begin{proof}
Let $\E$ be a minimal cluster. First of all, we apply Theorem~\ref{smalldiam} with $\sigma=5^{-N}$ and any balanced atlas $\{\varphi_i,\, U_i,\, V_i\}$ with oscillation smaller than $\sigma$, obtaining a constant $\bar\eps_1$  so that the claim of the Theorem holds. Then, we call $\bar r>0$ a constant such that, for every $p\in M$, there exists some index $i$ such that $p\in V_i$ and $B_p=B(\varphi_i^{-1}(p),\bar r)$ is compactly contained in $U_i$. By~\cite{PratelliScattaglia2024}, the so called $\eps-\eps$ property holds. More precisely, up to decrease the constant $\bar r$, we have two constants $\bar\eps_2 < \bar\eps_1$ and $C_5>0$ such that, for every $p \in M$ and every $\eps\in\R^m$ with $|\eps|<\bar\eps_2$, there exists another cluster $\E'$ with $\E'=\E$ in $\varphi_i(B_p)$, and
\begin{align}\label{PraSca}
|\E'|=|\E|+\eps\,, && P_M(\E') \leq P_M(\E) + C_5 |\eps|\,.
\end{align}
Without loss of generality, we can assume that
\begin{equation}\label{smallr}
\bar\eps_2< \bigg(\frac{\bar r}{4C_4}\bigg)^N \wedge \bigg(\frac{1-5^{-N}}{C_2^- \, C_5}\bigg)^N\, N^N\omega_N\,.
\end{equation}
We want to show that every connected component of $\E$ has total volume at least $\bar\eps_2$; of course, this will conclude the thesis, because as an obvious consequence the connected components will be at most $\tv\E/\bar\eps_2$. Let us then suppose that $\widetilde\E$ is a connected component of $\E$ with total volume $\tv{\widetilde\E}<\bar\eps_2$. Since $\bar\eps_2<\bar\eps_1$, by Theorem~\ref{smalldiam} we know that $\widetilde\E$ is entirely contained in some $V_i$, and we can call $\F=\varphi_i^{-1}(\widetilde\E) \comp U_i$. Moreover, by~(\ref{smallr})
\[
\diam(\F) \leq 4C_4\tv \F^{1/N} < 4 C_4 \bar\eps_2^{1/N} < \bar r\,,
\]
and then, calling $p$ an arbitrary point of $\widetilde\E$, $\F$ is compactly contained in the ball $B_p$ defined above. Calling $\eps=|\F|$, the $\eps-\eps$ property provides us with a cluster $\E'$ which coincides with $\E$ in $\varphi_i(B_p)$, and such that~(\ref{PraSca}) holds. In particular, $\widetilde\E$ is a connected component also for $\E'$. We can then define $\widehat\E= \E' \setminus \widetilde\E$, so that by construction and by~(\ref{PraSca}) we have $|\widehat\E|=|\E|$. Since $\E$ is a minimal cluster, this implies $P_M(\widehat\E)\geq P_M(\E)$. Thus, using~(\ref{PraSca}), we have
\begin{equation}\label{generic}
0\leq P_M(\widehat\E)-P_M(\E)=P_M(\E') - P_M(\widetilde\E)- P_M(\E)\leq C_5 |\eps| - P_M(\widetilde\E)\,.
\end{equation}
Let us now evaluate $P_M(\widetilde\E)$. Since the atlas is balanced, for every $i$ we can take a point $x_i\in U_i$ such that $\{v \in \R^N,\, h_i(x_i,v)\leq 1\}$ is balanced; and then, also recalling that the atlas has oscillation smaller than $\sigma=5^{-N}$, for every $x\in U_i$ and every $\nu\in\S^{N-1}$ we have $h_i(x,\nu)\geq (1-\sigma) h_i(x_i,\nu)\geq (1-\sigma)/C_2^-$. Consequently,
\[
P_M(\widetilde\E) \geq \frac{1-\sigma}{C_2^-}\, \H^{N-1}(\partial^*\F)
\geq \frac{1-\sigma}{C_2^-}\, N\omega_N^{1/N} \tv\F^{\frac{N-1}N}
=\frac{1-\sigma}{C_2^-}\, N\omega_N^{1/N} |\eps|^{\frac{N-1}N}\,,
\]
and inserting this estimate in~(\ref{generic}) we find
\[
|\eps|\geq \bigg(\frac{1-5^{-N}}{C_2^- \, C_5}\bigg)^N\, N^N\omega_N \,.
\]
Since this is in contradiction with~(\ref{smallr}), we have proved that a connected component of $\E$ with total volume less thann $\bar\eps_2$ cannot exist, thus the proof is completed.
\end{proof}

It is important to notice that the constant $\bar\eps_2$ of the above lemma in principle depends on the cluster $\E$. Therefore, the bound on the number of connected components is not uniform. However, this is enough to allow us to perform a homothethy argument to prove Theorem~\mref{main1}.

\begin{proof}[Proof of Theorem~\mref{main1}]
Let $\E$ be an $m$-minimal cluster. By Lemma~\ref{fmcc}, we know that $\E$ is done by finitely many connected components, that we call $\E_1,\, \E_2,\, \dots\,,\, \E_H$ with some $H\in\N$. Our goal is to show that $H\leq m$, provided that $\tv\E$ is smaller than a suitable constant $\bar\eps$, that we are going to precisely define in~(\ref{defbareps}), only depending on $M$ and $m$. First of all, we apply Theorem~\ref{smalldiam} with constant $\sigma=5^{-N}$ and an arbitrary balanced atlas $\{\varphi_i,\, U_i,\, V_i\}$ with oscillation smaller than $\sigma$, getting a constant $\bar\eps_1$ such that the claim of the theorem holds. If $\tv\E<\bar\eps_1$, then obviously every connected component $\E_\ell$ of $\E$ has total volume less than $\bar\eps_1$. As a consequence, for every $1\leq \ell\leq H$ there exists some $i=i(\ell)$ such that $\E_\ell$ is compactly contained in $V_i$, and thus we call $\F_\ell=\varphi_i^{-1}(\E_\ell)$.

Let $\lambda_\ell\in\R$, for $1\leq \ell\leq H$, be some constants, to be specified later. For $\delta\in (-1/2,1/2)$ and for each $1\leq \ell\leq H$, we can fix a point $x_\ell$ in $\F_\ell$ and define the modified cluster $\F_\ell^\delta\subseteq\R^N$ as
\[
\F_\ell^\delta = x_\ell +  \big( 1+\delta\lambda_\ell\big)^{1/N}\, (\F_\ell-x_\ell)\,.
\]
In other words, we keep the point $x_\ell$ fixed and we perform a homothety of ratio $(1+\delta\lambda_\ell)^{1/N}$. Since the components $\F_\ell$ are finitely many, if $|\delta|$ is small enough then each cluster $\F_\ell^\delta$ is still compactly contained in $U_i$, and the clusters $\E^\delta_\ell=\varphi_i(\F_\ell^\delta)$ are still pairwise disjoint. Therefore, we can call $\E^\delta=\cup_{\ell=1}^H \E_\ell^\delta$, and the clusters $\E_\ell^\delta$ are the connected components of $\E^\delta$. The volume of $\E^\delta$ is easily calculated as
\[
\big|\E^\delta\big|_M = \sum\nolimits_{\ell=1}^H |\F_\ell^\delta| 
= \sum\nolimits_{\ell=1}^H \big(1+\delta\lambda_\ell) |\F_\ell|
= |\E|_M + \delta \Big( \sum\nolimits_{\ell=1}^H \lambda_\ell |\F_\ell|\Big)\,.
\]
Now, keep in mind that the $|\F_\ell|$'s are $H$ vectors in $\R^m$. As a consequence, if $H>m$ there exists a nontrivial choice of $\lambda_\ell$ so that
\[
\sum\nolimits_{\ell=1}^H \lambda_\ell |\F_\ell| = 0\,.
\]
With this choice, for every $|\delta|\ll 1$ the cluster $\E^\delta$ has the same volume as $\E$, so it can be used as a competitor. Let us now estimate the perimeter of $\E$. To do so, for every $1\leq \ell\leq H$ we call $E_{\ell,j}$ for $1\leq j\leq m$ the components of $\E_\ell$, and similarly $E^\delta_{\ell,j}$ are the components of $\E^\delta_\ell$. In other words, we write $\E_\ell=(E_{\ell,1},\, E_{\ell,2},\, \dots\,,\, E_{\ell,m})$ and $\E^\delta_\ell=(E^\delta_{\ell,1},\, E^\delta_{\ell,2},\, \dots\,,\, E^\delta_{\ell,m})$. By definition, writing $E_{\ell,0}$ (resp., $E^\delta_{\ell,0}$) the union of the sets $E_{\ell,j}$ (resp., $E^\delta_{\ell,j}$) for $1\leq j\leq m$, we have then
\begin{equation}\label{recall1}
P_M\big(\E^\delta_\ell) =  \frac 12\, \sum_{j=0}^m P_M(E^\delta_{\ell,j})\,,
\end{equation}
and by construction each $E^\delta_{\ell,j}$ satisfies
\begin{equation}\label{homothety}
E^\delta_{\ell,j} = x_\ell + \big( 1+\delta\lambda_\ell\big)^{1/N}\, (E_{\ell,j}-x_\ell)\,.
\end{equation}
Calling then for simplicity of notation $\nu(x)$ the unit outer normal vector to $E^\delta_{\ell,j}$ at $x\in \partial^* E^\delta_{\ell,j}$, with the change of variables $x=\tau(y)=x_\ell+( 1+\delta\lambda_\ell)^{1/N}(y-x_\ell)$ we have
\[
P_M(E^\delta_{\ell,j}) = \int_{\partial^* E^\delta_{\ell,j}} h_i(x,\nu(x))\, d\H^{N-1}(x)
=\big( 1+\delta\lambda_\ell\big)^{\frac{N-1}N}\, \int_{\partial^* E_{\ell,j}} h_i\big(\tau(y),\nu((\tau(y))\big) \, d\H^{N-1}(y)\,.
\]
Notice now that, by~(\ref{homothety}), the unit normal vector $\nu(y)$ to $E_{\ell,j}$ at $y\in \partial^* E_{\ell,j}$ coincides with $\nu(\tau(y))$. As a consequence, if we define
\[
\gamma(y) = h_i\big(\tau(y),\nu((\tau(y))\big) - h_i\big(y,\nu(y)\big)
= h_i\big(\tau(y),\nu(y)\big) - h_i\big(y,\nu(y)\big)\,,
\]
the above equality gives
\begin{equation}\label{starthere}
P_M(E^\delta_{\ell,j})=\big( 1+\delta\lambda_\ell\big)^{\frac{N-1}N}\, \int_{\partial^* E_{\ell,j}} h_i(y,\nu(y))+\gamma(y) \, d\H^{N-1}(y)\,.
\end{equation}
We need now to expand the above term up to the second order for $\delta\approx 0$. We begin with the function $\tau$: by definition,
\[
\tau(y) = y + \frac {\lambda_\ell}N\, (y-x_\ell) \delta + \frac {(1-N)\lambda_\ell^2}{2N^2}\, (y-x_\ell)\delta^2 + o(\delta^2)\,.
\]
Let us now pass to the expansion of $\gamma$: writing for brevity $\psi(z)=\psi_{i,\nu}(z)=h_i(z,\nu)$, we have
\[\begin{split}
\gamma(y) &=h_i(\tau(y),\nu(y)) - h_i(y,\nu(y))=\psi(\tau(y))-\psi(y)\\
&= \nabla \psi(y) \cdot (\tau(y) - y) + \frac 12\, (\tau(y)-y)^t D^2\psi(y) (\tau(y)-y)\\
&= \nabla \psi(y) \cdot \bigg(\frac {\lambda_\ell}N\, (y-x_\ell) \delta+ \frac {(1-N)\lambda_\ell^2}{2N^2}\, (y-x_\ell)\delta^2\bigg)\\
& \phantom{= \nabla \psi(y) \cdot \bigg(\frac {\lambda_\ell}N\, (y-x_\ell) \delta\,}+\frac {\lambda_\ell^2}{2N^2}\, (y-x_\ell)^t D^2\psi(y)  (y-x_\ell)\,\delta^2+o(\delta^2)\,.
\end{split}\]
Finally, we observe that
\[
(1+\delta\lambda_\ell)^{\frac{N-1}N} = 1 + \frac{N-1}N\, \lambda_\ell\delta -\frac{N-1}{N^2}\, \lambda_\ell^2\delta^2 + o (\delta^2)\,.
\]
Inserting the last two estimates in~(\ref{starthere}), we obtain
\begin{equation}\label{recall2}
P_M(E^\delta_{\ell,j}) = P_M(E_{\ell,j}) + \alpha \delta + \beta \delta^2 + o(\delta^2)\,,
\end{equation}
with two suitable constants $\alpha$ and $\beta$. We are not interested in evaluating $\alpha$; concerning $\beta$, instead, we have
\begin{equation}\label{hereisbeta}\begin{split}
\beta=\frac{\lambda_\ell^2}{N^2} \int_{\partial^* E_{\ell,j}} -(N-1)h_i(y,\nu(y)) &+ \frac{N-1}2\,\nabla\psi_{i,\nu(y)}(y)\cdot (y-x_\ell)\\
&+\frac 12\,(y-x_\ell)^t D^2\psi_{i,\nu(y)}(y)  (y-x_\ell)\,d\H^{N-1}(y)\,.
\end{split}\end{equation}
Now, since the atlas is balanced and its oscillation is at most $\sigma=5^{-N}$, for every $z\in U_i$ and every unit vector $\nu\in\S^{N-1}$ we have
\[
h_i(z,\nu) \geq \frac{1-\sigma}{C_2^-}\geq \frac 1{2C_2^-}\,.
\]
On the other hand, keep in mind that $x_\ell$ is a fixed point in $\F_\ell$, and by Theorem~\ref{smalldiam} we know that $\diam(\F_\ell)\leq 4 C_4\tv{\F_\ell}^{1/N}$, so in particular $\diam(\F_\ell)\leq 4 C_4\bar\eps^{1/N}$ since we are assuming that $\tv\E\leq\bar\eps$. But then, for every $y\in\partial^* E_{\ell,j}$ we have $|y-x_\ell| \leq 4 C_4 \bar\eps^{1/N}$, and then, since the definition~(\ref{defbareps}) of $\bar\eps$ in particular will satisfy $\bar\eps\leq (4C_4)^{-N}$, we get
\[\begin{split}
\bigg|\frac{N-1}2\,\nabla\psi_{i,\nu(y)}(y)\cdot (y-x_\ell)&+\frac 12\,(y-x_\ell)^t D^2\psi_{i,\nu(y)}(y)  (y-x_\ell)\bigg|\\
&\leq \bigg(\frac {N-1}2\, \|\nabla\psi_{i,\nu(y)}\|_{L^\infty} + \frac 12\, \|D^2\psi_{i,\nu(y)}\|_{L^\infty} \bigg) \,4C_4\bar\eps^{1/N}
=C_6\bar\eps^{1/N}\,,
\end{split}\]
where
\[
C_6 := 4 C_4  \max_{i,\, \nu}\bigg\{\bigg(\frac {N-1}2\, \|\nabla\psi_{i,\nu}\|_{L^\infty} + \frac 12\, \|D^2\psi_{i,\nu}\|_{L^\infty} \bigg) \bigg\}\,.
\]
Notice that the constant $C_6$ is well defined by compactness (and since the manifold is of class ${\rm C}^2$), and it only depends on the atlas, thus only on the manifold and on its dimension $N$, since the atlas was generic. Inserting the above estimates in~(\ref{hereisbeta}) gives
\[
\beta \leq \frac{\lambda_\ell^2}{N^2} \int_{\partial^* E_{\ell,j}} -\frac{N-1}{2C_2^-} + C_6\bar\eps^{1/N}\,d\H^{N-1}(y)\,.
\]
We are finally in position to define the constant $\bar\eps$ as
\begin{equation}\label{defbareps}
\bar\eps := \min \bigg\{ \bar\eps_1,\, \bigg(\frac 1{4C_4}\bigg)^N,\,\bigg(\frac{N-1}{3C_2^- C_6}\bigg)^N\bigg\}\,,
\end{equation}
which by construction only depends on the manifold $M$ and its dimension $N$. With this choice, we have that the constant $\beta$ is negative, and actually strictly negative unless $\lambda_\ell=0$. As a consequence, in a suitably small neighborhood of $\delta=0$ the function $\delta\mapsto P_M(\E_\ell^\delta)$ is strictly concave whenever $\lambda_\ell\neq 0$ (while otherwise it is constant). Since the choice of the coefficients $\{\lambda_\ell\}$ is nontrivial, we deduce that the perimeter of the cluster $\E^\delta$ is a strictly concave function of $\delta$ in a small neighborhood of $\delta=0$. Since all the clusters $\E^\delta$ have the same volume and $\E=\E^0$ is a minimal cluster, this is impossible. The contradiction show that necessarily $H\leq m$, and this concludes the first part of the thesis.

Concerning the validity of~(\ref{bounddiam}), we have already observed above that, thanks to Theorem~\ref{smalldiam}, we have $\diam(\F_\ell)\leq 4 C_4\tv{\F_\ell}^{1/N}$. We can then argue exactly as in the proof of Corollary~\ref{corsmalldiam}, namely, observing that the ratio between the distance on $M$ and the distance measured in the finite charts of the atlas is bounded from above and below.
\end{proof}

\section{The Riemannian case: proof of Theorem~\mref{main2}\label{secma2}}

This section is devoted to prove Theorem~\mref{main2}, which says that sufficiently small minimal clusters are actually connected if the manifold is a Riemannian one. Many of our arguments work whenever $M$ is a fixed-norm manifold; as a consequence, we will make this assumption through all this section (and we will underline when we are going to use that $M$ is actually Riemannian). Thus, there is a given balanced norm $\|\cdot\|$ in $\R^N$ which coincides, up to a linear transformation with unit-determinant, with the infinitesimal norm at each point $p\in M$. For brevity of notation, we will call $P_*$ the perimeter in $\R^N$ associated to this norm, that is, for every set $F\subseteq\R^N$ we define
\[
P_*(F)  = \int_{\partial^* F} \|\nu_F(x)\| \, d\H^{N-1}(x)\,,
\]
and for every cluster $\F\subseteq\R^N$ we define $P_*(\F)$ as usual via~(\ref{perclu}). We call then $J:(\R^+)^m\to\R^+$ the isoperimetric function corresponding to this perimeter, that is, for every $v\in(\R^+)^m$ we set
\[
J(v) = \min \Big\{ P_*(\F),\, \F\subseteq\R^N,\, |\F|=v\Big\}\,.
\]
Notice that, by the homogeneity of the perimeter, for every $\lambda>0$ and every $v\in(\R^+)^m$ we have $J(\lambda v) = \lambda^{N-1} J(v)$, thus it is enough to study the function $J$ on the vectors with total norm equal to $1$. We start now with the following easy observation.
\begin{lemma}\label{largesmall}
For every small $\delta>0$ there exists $\mu>0$ such that, for any cluster $\F\subseteq\R^N$, if the largest connected component of $\F$ has total volume less than $(1-\delta)\tv\F$, then
\[
P_*(\F) \geq (1+\mu) \, J(|\F|)\,.
\]
\end{lemma}
\begin{proof}
As noticed above, the homogeneity allows to restrict ourselves to clusters with total volume equal to $1$. Hence, we let $v\in(\R^+)^m$ be any vector with $|v|_1=\sum_{i=1}^m v_i=1$. By standard concentration and compactness arguments, we know that there exist minimal clusters in $\R^N$ with volume $v$, and that each such minimal cluster is connected. By continuity and compactness, the minimal perimeter of clusters $\F$ with $|\F|=v$ and with the largest connected component having total volume less than $1-\delta$ is strictly greater than $J(v)$, thus we can write it as $(1+\mu)\, J(v)$ for some $\mu>0$ depending on $v$ (and, of course, on $\delta$). The claim is equivalent to say that we can find a suitable constant $\mu$, only depending on $\delta$ and on the norm, but not on $v$.

Let us assume by contradiction that this is false. Then, we can find a sequence of volumes $v_n$ with $|v_n|_1=1$, and a sequence of clusters $\F_n$ with $|\F_n|=v_n$, in such a way that for every $n$ we have $\F_n=\F_n'\cup \F_n''$, being $\F_n'$ the largest connected component of $\F_n$, and so that for every $n$ one has
\begin{align}\label{muto0}
\tv{\F_n'}\leq 1-\delta\,, && P_*(\F_n)\leq \bigg( 1+ \frac 1n\bigg)\, J(v_n)\,.
\end{align}
Without loss of generality, we can assume that $\F_n'$ is a minimal cluster with volume $|\F_n'|$ and $\F_n''$ is a minimal cluster with volume $|\F_n''|$, in particular they are both connected and moreover $\tv{\F_n'}\geq 1/2$. Calling $d$ the maximal diameter of $M$ (in the sense of Definition~\ref{defmadi}), we have that
\begin{align*}
\diam(\F_n') \leq d \tv{\F_n'}^{\frac 1N} \leq d\,, && \diam(\F_n'') \leq d \tv{\F_n''}^{\frac 1N} \leq d\,.
\end{align*}
Since the ambient space is $\R^N$ with a fixed norm, translations leave volume and perimeter invariant; therefore, keeping in mind that $\F_n'$ and $\F_n''$ are different connected components of $\F_n$, by the above estimate on the diameters we can assume without loss of generality that
\begin{align}\label{staythere}
\F_n' \comp B(0,d) \,, && \F_n'' \comp B(Q,d)\,,
\end{align}
where $Q$ is any point in $\R^N$ with distance larger than $2d$ from the origin, for instance we can take $Q=3d{\rm e}_1$. Up to pass to a subsequence, we can assume that $v_n \to v$, of course with $|v|_1=1$, and that the clusters $\F_n'$ and $\F_n''$ converge weakly* in the BV sense to some clusters $\F'$ and $\F''$. Since the convergence holds in particular in $L^1_{\rm loc}$, by~(\ref{staythere}) and~(\ref{muto0}) we know that
\begin{align*}
\F' \subseteq B(0,d) \,, && \F'' \subseteq B(Q,d)\,, && |\F'\cup\F''| = v,\, && \tv{\F'} = \lim_{n\to\infty} \tv{\F_n'} \in \big[ 1/2 , 1-\delta]\,.
\end{align*}
Moreover, by the lower semicontinuity of the perimeter, again using~(\ref{muto0}) and the obvious fact that $J$ is continuous we have
\[
P_* (\F) \leq \liminf_{n\to\infty} P_*(\F_n)  \leq \liminf_{n\to\infty} \bigg( 1+ \frac 1n\bigg)\, J(v_n) = J(v)\,.
\]
In other words, $\F$ is a minimal cluster of volume $v$. Since $\F$ is not connected, this is impossible, and we have found the searched contradiction.
\end{proof}

The above lemma allows us to present the first step in the proof of Theorem~\mref{main2}, that is, we exclude that arbitrarily small clusters can have a largest connected component for which the percentage of the total volume is not going to $1$. More precisely, the result is the following.
\begin{lemma}\label{easycase}
Let $M$ be as in Theorem~\mref{main2}. For every $\delta>0$ there exists some $\eps_\delta>0$ such that for any minimal cluster $\E$ with $\tv\E<\eps_\delta$ there is a connected component $\E_1$ of $\E$ satisfying $\tv{\E_1}\geq (1-\delta) \tv\E$.
\end{lemma}
\begin{proof}
Let $\mu=\mu(\delta)$ be given by Lemma~\ref{largesmall}, and let us fix a balanced atlas with oscillation smaller than
\begin{equation}\label{sigmanow}
\sigma = 4^{-N} \wedge \frac \mu 4\,.
\end{equation}
Up to a change of variables, we can assume that for every $i$ there exists a point $z_i\in U_i$ such that the norm $\|\cdot\|_{i,z_i}$ coincides with the representative $\|\cdot\|$ of the fixed norm that we have chosen. If $E\subseteq M$ is any set compactly contained in some $V_i$, then by~(\ref{choiceatlas}) we have
\[
P_M(E) = \int_{\partial^* F} h_i(x,\nu_F(x))\, d\H^{N-1}(x)\geq (1-\sigma) \int_{\partial^* F} h_i(z_i,\nu_F(x))\, d\H^{N-1}(x) = (1-\sigma) P_*(F)\,,
\]
where we have denoted $F=\varphi_i^{-1}(E)$. As a consequence, if $\widetilde\E$ is any $m$-cluster which is compactly contained in some $V_i$ and we denote by $\widetilde\F=\varphi_i^{-1}(\widetilde\E)$, then we have
\begin{equation}\label{balfixpo}
P_M(\widetilde\E) \geq (1-\sigma) P_*(\widetilde \F)\,.
\end{equation}
Let us now apply Theorem~\ref{smalldiam} to our balanced atlas, getting a constant $\bar\eps_1$, which depends on $M,\, N$ and $\delta$. Let now $\eps_\delta\leq \bar\eps_1$ be a constant, again only depending on $M,\,N$ and $\delta$, to be specified later, and let $\E$ be a minimal cluster with volume $v=|\E|$ so that the total volume satisfies $\tv\E = |v|_1 <\eps_\delta$. By Lemma~\ref{fmcc}, we know that $\E$ is done by finitely many connected components $\E_1,\, \E_2,\, \dots\,,\, \E_H$ (up to replace $\eps_\delta$ by a possibly smaller constant we could apply Theorem~\mref{main1} to be sure that $H\leq m$, but this would not make any difference in the proof). To obtain the proof, we assume that $\tv{\E_1}<(1-\delta)|v|_1$, being $\E_1$ the connected component with the largest total volume, and we look for a contradiction.\par

For every $1\leq \ell \leq H$ there exists an index $i=i(\ell)$ such that $\E_\ell$ is compactly contained in $V_i$; thus, we can define the cluster $\F_\ell = \varphi_i^{-1}(\E_\ell)\comp U_i\subseteq\R^N$. Up to translations, we can assume that all the clusters $\F_\ell$ are a positive distance apart from each other, and call $\F = \cup_{\ell=1}^H \F_\ell$, which satisfies $|\F|=|\E|=v$. Since $\F_1$ is the connected component of $\F$ with the largest total volume, and $\tv{\F_1}=\tv{\E_1}<(1-\delta)|v|_1=(1-\delta)\tv\F$, by Lemma~\ref{largesmall} we have that
\[
P_*(\F) \geq (1+\mu) J(v)\,.
\]
Applying~(\ref{balfixpo}) to any of the clusters $\E_\ell$, we have
\begin{equation}\label{cararrives}
P_M(\E) = \sum_{\ell=1}^H P_M(\E_\ell) 
\geq (1-\sigma) \sum_{\ell=1}^H P_*(\F_\ell) = (1-\sigma) P_*(\F) \geq (1-\sigma)(1+\mu) J(v)\,.
\end{equation}
Now, we let $\G\subseteq\R^N$ be a minimal cluster of volume $v$ in $\R^N$ with the norm $\|\cdot\|$. By rescaling, we have $\diam(\G)\leq d (\tv\G)^{1/N} \leq d \eps_\delta^{1/N}$, where $d$ is the maximal diameter for $M$ (see Definition~\ref{defmadi}), Up to replace $\eps_\delta$ by a smaller constant (depending only on the atlas, thus on $M$, $N$ and $\delta$), we can then assume that $\G$ is compactly contained in the chart $U_1$, up to a translation. Then, we can call $\widetilde\E=\varphi_1(\G)$, and notice that by construction $\widetilde\E\subseteq M$ is a cluster with volume $|\widetilde\E|=v=|\E|$. Arguing as in~(\ref{balfixpo}) using again that the atlas has oscillation smaller than $\sigma$, we have then
\[
P_M(\widetilde\E)\leq (1+\sigma) P_*(\G) =(1+\sigma) J(v)\,.
\]
Putting this estimate together with~(\ref{cararrives}), and recalling that $\E$ is a minimal cluster, we have then
\[
P_M(\E) \leq P_M(\widetilde\E) \leq (1+\sigma) J(v)
\leq \frac{1+\sigma}{(1-\sigma)(1+\mu)} \, P_M(\E)\,,
\]
which implies
\[
1+\sigma \geq (1-\sigma)(1+\mu)\,.
\]
Since this inequality is not compatible with the definition~(\ref{sigmanow}), we have found the searched contradiction, and this concludes the proof.
\end{proof}

The above lemma says that minimal clusters with extremely small total volume must be ``almost connected'', in the sense that the percentage of the total volume which does not belong to the largest connected component must be arbitrarily small if the total volume is very small. To prove Theorem~\mref{main2} we have then also to exclude the existence of very small (in percentage) connected components, and this will be the hardest part of the proof. We start with a technical result, which says that if a set $E$ contained in an arbitrarily high prism is flat up to a very small error, then the same is true in a smaller prism with a fixed height, not depending on the error. More precisely, our result is the following.
\begin{lemma}\label{qutr}
There exists a purely dimensional constant $H=H(N)$ such that the following holds. Let $Q=(-1/2,1/2)^{N-1}\subseteq \R^{N-1}$, let $W>1$ and $\rho\ll 1$ be two constants, and let $E\subseteq Q\times (-W,W)$ be a set satisfying
\begin{gather}
\H^{N-1}\Big( \partial^*E \cap \big(Q\times (-W,W)\big)\Big) \leq 1+\rho \,,\label{cond1}\\
\H^{N-1}\Big( E \cap \big(Q\times \{-W\}\big)\Big) \geq 1-\rho\,,\label{cond2}\\
\H^{N-1}\Big( E \cap \big(Q\times \{W\}\big)\Big) \leq \rho\,.\label{cond3}
\end{gather}
Then, there exists an interval $(a,b)\subseteq (-W,W)$ with length at most $H$ such that conditions~(\ref{cond1})--(\ref{cond3}) are satisfied with $(a,b)$ in place of $(-W,W)$ and $Q\times \{a\}$ (resp., $Q\times \{b\}$) in place of $Q\times \{-W\}$ (resp., $Q\times \{W\}$).
\end{lemma}
\begin{proof}
For every $t\in [-W,W]$, we call for brevity
\[
v(t) = \H^{N-1}\Big( E \cap \big(Q \times \{t\}\big)\Big)\,.
\]
Notice that for $\H^1$-a.e. $t\in [-W,W]$, by Vol'pert Theorem (see for instance~\cite{MR216338,MR1857292}) we have that $v(t)$ is well-defined, and moreover
\[
\partial^* \Big(E \cap \big( Q\times \{t\}\big)\Big) = \partial^* E \cap \big( Q\times \{t\}\big)
\]
up to $\H^{N-2}$-negligible subsets. As a consequence, by the relative isoperimetric inequality in the cube we have
\begin{equation}\label{byvolp}
\H^{N-2} \Big(\partial^* E \cap \big( Q\times \{t\}\big)\Big) \geq C_N \Big( v(t)\wedge \big(1-v(t)\big)\Big)^{\frac{N-2}{N-1}}\,,
\end{equation}
where $C_N$ is a constant only depending on $N$. Let us now define
\begin{align}\label{tobegen}
s_1^- = \sup \Big\{ t \in [-W,W]:\, v(t) \geq 1-\rho\Big\}\,, && s_1^+ = \inf \Big\{ t \in [-W,W]:\, v(t) \leq \rho\Big\}\,,
\end{align}
which are well defined thanks to~(\ref{cond2}) and~(\ref{cond3}). Notice that, by projection, we have
\begin{equation}\label{projarg}
\H^{N-1}\Big( \partial^* E \cap\big( Q \times [s_1^-,W) \big)\Big) \geq \big| v(W) - \limsup_{s\nearrow s_1^-} v(s)\big| \geq  1-2\rho\,,
\end{equation}
and similarly
\[
\H^{N-1}\Big( \partial^* E \cap\big( Q \times (-W,s_1^+] \big)\Big) \geq 1-2\rho\,.
\]
Since $\rho<1/5$, by~(\ref{cond1}) we deduce that $s_1^-< s_1^+$. In the same way, we can define
\begin{align*}
s_j^- = \sup \Big\{ t \in [-W,W]:\, v(t) \geq 1-2^j \rho\Big\}\,, && s_j^+ = \inf \Big\{ t \in [-W,W]:\, v(t) \leq 2^j \rho\Big\}\,,
\end{align*}
and we still have that $s_j^-< s_j^+$ as soon as $2^j \rho< 1/4$. Calling then $\jm$ the largest integer such that $2^{\jm} \rho<1/4$, by construction we have $s_1^-\leq s_2^-\leq \cdots \leq s_\jm^-\leq s_\jm^+\leq \cdots \leq s_2^+\leq s_1^+$. Notice now that, for almost every $t\in I_j:=(s_j^-, s_{j+1}^-)\cup (s_{j+1}^+,s_j^+)$, by~(\ref{byvolp}) we have
\[
\H^{N-2} \Big(\partial^* E \cap \big( Q\times \{t\}\big)\Big) \geq C_N \big( 2^j \rho\big)^{\frac{N-2}{N-1}}\,.
\]
As a consequence, calling $l_j=\H^1(I_j)$, by the Coarea Formula we deduce
\[
\H^{N-1}\Big(\partial^* E \cap \big(Q \times I_j  \big)\Big)
\geq \int_{I_j} \H^{N-2} \Big(\partial^* E \cap \big( Q\times \{t\}\big)\Big)\, dt\geq C_N l_j \big( 2^j \rho\big)^{\frac{N-2}{N-1}}\,.
\]
Moreover, arguing by projection as in~(\ref{projarg}), we get
\[
\H^{N-1}\Big( \partial^* E \cap\big( Q \times [s_{j+1}^-,s_{j+1}^+] \big)\Big) \geq 1-2^{j+2}\rho\,.
\]
Putting the last two estimates together with~(\ref{cond1}), we get
\[
1+\rho\geq 1-2^{j+2}\rho + C_N l_j \big( 2^j \rho\big)^{\frac{N-2}{N-1}}\,,
\]
which can be rewritten as
\[
l_j\leq \frac{2^{-j\,\frac{N-2}{N-1}} (2^{j+2}+1)}{C_N}\,\rho^{\frac 1{N-1}}
\leq \frac 8{C_N}\,2^{\frac j{N-1}} \,\rho^{\frac 1{N-1}}\,.
\]
As a consequence, we have
\begin{equation}\label{bigones}\begin{split}
s_1^+-s_1^- &= s_\jm^+ - s_\jm^- + \sum_{j=1}^\jm l_j
\leq s_\jm^+ - s_\jm^- + \frac{16}{C_N (2^{\frac 1{N-1}}-1)}\,(2^\jm\rho)^{\frac 1{N-1}}\\
&\leq s_\jm^+ - s_\jm^- + \frac{16}{C_N 4^{\frac 1{N-1}} (2^{\frac 1{N-1}}-1)}\,.
\end{split}\end{equation}
Moreover, by definition, for almost every $s_\jm^-<t<s_\jm^+$ we have
\[
\frac 18 \leq 2^\jm\rho < v(t)< 1-2^\jm\rho \leq \frac 78\,,
\]
which again by~(\ref{byvolp}) implies
\[
\H^{N-2} \Big(\partial^* E \cap \big( Q\times \{t\}\big)\Big) \geq \frac{C_N}{8^{\frac{N-2}{N-1}}} \,,
\]
and then
\[
2 \geq \H^{N-1}\Big( \partial^* E \cap\big( Q \times [s_\jm^-,s_\jm^+] \big)\Big)
\geq \frac{C_N}{8^{\frac{N-2}{N-1}}}\, \big(s_\jm^+-s_\jm^-\big)\,,
\]
which can be rewritten as
\[
s_\jm^+-s_\jm^-\leq \frac{2\cdot 8^{\frac{N-2}{N-1}}}{C_N}\,.
\]
Putting together this estimate and~(\ref{bigones}), we get a purely dimensional constant $H=H(N)$ such that $s_1^+-s_1^-<H$. By definition~(\ref{tobegen}), we can find $a$ and $b$ sufficiently close to $s_1^-$ and $s_1^+$ so that $b-a<H$ and that $v(a)\geq 1-\rho$ and $v(b)\leq \rho$. This gives the validity of~(\ref{cond2}) and~(\ref{cond3}) with $a$ and $b$ in place of $-W$ and $W$. Since the validity of~(\ref{cond1}) with $(a,b)$ in place of $(-W,W)$ is obvious, the thesis is concluded.
\end{proof}

The importance of the above technical result in our paper relies in the following consequence.

\begin{prop}\label{flattening}
For every $\theta>0$ there exists $\rho\leq \theta$, only depending on $\theta$ and on $N$, such that the following holds. Let $E\subseteq\R^N$ be a set, and assume that for some $R>r>0$, calling $Q_r=(-r/2,r/2)^{N-1}$ and $Q_r^R=Q_r\times (-R/2,R/2)$ one has
\begin{gather}
\H^{N-1}\Big(\partial^* E\cap Q_r^R\Big) \leq (1+\rho) r^{N-1}\,,\label{hyp1} \\
\H^N\Big(\big(E \cap Q_r^R\big)\Delta \big(Q_r\times(-R/2,0)\big)\Big) \leq \rho^2\, r^{N-1} R\label{hyp2}\,.
\end{gather}
Then, there exist $\bar t\in [-\rho R-\theta r,\rho R+\theta r]$ and $r/2\leq r^-\leq r$ such that, calling
\begin{equation}\label{defF}
F= \big(E \setminus Q_{r^-}^R\big) \cup \big(Q_{r^-} \times (-R/2,\bar t)\big)\,,
\end{equation}
one has
\begin{align}\label{fincl}
\H^N\big(E \Delta F\big)\leq \theta r^N\,, && \H^{N-1}\big( \partial^* F \cap \partial Q_{r^-}^R\big)\leq \theta\, r^{N-1}\,, && |F|=|E|\,.
\end{align}
\end{prop}
\begin{proof}
By rescaling, we can reduce ourselves to consider the case $r=1$, which makes the notation simpler. For each $n\in\N$, let us then define $\rho_n=1/n$, and let $\{E_n\}$ be any set such that~(\ref{hyp1}) and~(\ref{hyp2}) hold true with constant $\rho_n$ in place of $\rho$, $Q=Q_1$ in place of $Q_r$, and some $R_n$ in place of $R$. To show the thesis, it is enough to prove the validity of~(\ref{fincl}) for $n$ large enough, up to a subsequence.

Let $n$ be fixed. Since by~(\ref{hyp2}) and by Fubini Theorem
\[
\rho_n^2 R_n \geq \H^N\Big(\big(E_n \cap Q^{R_n}\big)\Delta \big(Q\times(-R_n/2,0)\big)\Big)
\geq \int_0^{\rho_n R_n} \H^{N-1} \Big(E_n\cap  \big(Q\times \{t\}\big)\Big) \, dt\,,
\]
there exists some $0<R_n^+<\rho_n R_n $ such that
\[
\H^{N-1}\Big( E_n \cap \big(Q\times \{R_n^+\}\big)\Big) \leq \rho_n\,,
\]
and analogously there exists some $-\rho_n R_n < R_n^-<0$ such that
\[
\H^{N-1}\Big( E_n \cap \big(Q\times \{R_n^-\}\big)\Big) \geq 1-\rho_n\,.
\]
Therefore, we can apply Lemma~\ref{qutr} with the set $E_n$ and the interval $(R_n^-,R_n^+)$, and we get an interval $(a_n,b_n)$ of length at most $H$ such that conditions~(\ref{cond1})--(\ref{cond3}) are satisfied for the set $E_n$ and the interval $(a_n,b_n)$, with $\rho=\rho_n$. We now define the set $\widetilde E_n$ as
\[
\widetilde E_n = \Big\{ (x',x_N-a_n),\, (x',x_N)\in E_n,\, x'\in Q,\, x_N\in (a_n,b_n)\Big\}\cup \big(Q\times (-1,0)\big)\,.
\]
In words, $\widetilde E_n$ is obtained by taking the part of $E_n\cap \big(Q\times (-R_n,R_n)\big)$ relative to the interval $(a_n,b_n)$, translating it vertically so that the height $\{a_n\}$ corresponds to $0$, and then adding the cube $Q\times (-1,0)$. We will consider the set $\widetilde E_n$ inside the box $D=Q\times (-1,H+1)$. By construction, also keeping in mind that the last coordinate of any point of $\widetilde E_n$ is at most $b_n-a_n\leq H$, we have
\[\begin{split}
P_{\rm Eucl}\big( \widetilde E_n; D\big) = \H^{N-1}\Big(\partial^* E_n\cap \big(Q\times (a_n,b_n)\big)\Big) &+ \H^{N-1}\Big(E_n \cap \big(Q\times \{b_n\}\big)\Big)\\
& + 1- \H^{N-1}\Big(E_n \cap \big(Q\times \{a_n\}\big)\Big)\,,
\end{split}\]
where the first part is the perimeter of $\widetilde E_n$ inside $Q\times (0,b_n-a_n)$, and the second and third part are the perimeter at the level $b_n-a_n$ and $0$ respectively. Keeping in mind that for the set $E_n$ the estimate~(\ref{hyp1}) holds, as well as~(\ref{cond1})--(\ref{cond3}) with $a_n$ and $b_n$, the above inequality gives
\begin{equation}\label{smallper}
P_{\rm Eucl}\big( \widetilde E_n; D\big) \leq 1 + 3\rho_n\,.
\end{equation}
The sets $\widetilde E_n$ are then a sequence of sets of bounded perimeter in the bounded box $D$. There exists then a subsequence such that $\Chi{\widetilde E_n}$ converges in the weak BV sense to $\Chi{E}$ for some set $E\subseteq D$. Since the convergence is in particular strong in $L^1$, we deduce that
\begin{align}\label{alnst}
E\cap \big(Q\times (H,H+1)\big)=\emptyset\,, && Q\times (-1,0)\subseteq E\,.
\end{align}
By lower semicontinuity of the perimeter and by~(\ref{smallper}) we obtain that $P_{\rm Eucl}(E;D) \leq 1$, which together with~(\ref{alnst}) implies that $E$ is actually of the form
\[
E = Q \times (-1,t)
\]
for some $t\in [0,H]$. As a consequence, for $n$ large enough we have
\[
\frac \theta {5\cdot 2^N}>\H^N\Big(\widetilde E_n \Delta \big(Q\times (-1,t)\big)\Big)\geq \H^N\Big(\big(\widetilde E_n \Delta \big(Q\times (-1,t)\big)\big) \cap \big(Q\times (0,t\wedge (b_n-a_n)\big)\Big)\,.
\]
Keeping in mind the definition of $\widetilde E_n$, and setting $\hat t_n = a_n + t\wedge (b_n-a_n)\in (a_n,b_n]$, we deduce
\begin{equation}\label{internal}
\H^N\bigg(\Big( E_n \Delta \big(Q \times (a_n,\hat t_n) \big)\Big) \cap\big(Q \times (a_n,b_n) \big)\bigg) < \frac \theta {5\cdot 2^N}\,.
\end{equation}
Since by construction we have 
\begin{align*}
\H^{N-1}\Big( E_n \cap \big(Q\times \{a_n\}\big)\Big) \geq 1-\rho_n\,, && \H^{N-1}\Big( E_n \cap \big(Q\times \{b_n\}\big)\Big) \leq \rho_n\,,
\end{align*}
we deduce that
\[
\H^{N-1}\Big(\partial^* E_n \cap \big(Q \times (a_n,b_n)\big) \Big) \geq 1-2\rho_n\,,
\]
which by~(\ref{hyp1}) implies
\begin{equation}\label{star}
\H^{N-1}\bigg(\partial^* E_n \cap \Big(Q \times\big( (-R_n/2,a_n)\cup (b_n,R_n/2)\big)\Big) \bigg) \leq 3\rho_n\,.
\end{equation}
In particular, $P_{\rm Eucl}(E_n; Q\times (-R_n/2,a_n))\leq 3\rho_n$, and by the relative isoperimetric inequality this implies that the volume of $E_n\cap (Q\times (-R_n/2,a_n))$ is either smaller than $\rho_n$ or larger than $a_n+R_n/2-\rho_n$ (keep in mind that $\rho_n^{\frac N{N-1}}\gg \rho_n$ for $n$ large enough). By~(\ref{hyp2}) and the fact that $\rho_n\ll 1$, we exclude the first possibility; in other words, we have obtained
\[
\H^N\Big( \big(Q \times (-R_n/2,a_n)\big)\setminus E_n\Big) < \rho_n\,.
\]
In the very same way, $P_{\rm Eucl}(E_n; Q\times (b_n,R_n/2))\leq 3\rho_n$, and by the relative isoperimetric inequality we deduce that
\[
\H^N\Big( \big(Q \times (b_n,R_n/2)\big)\cap E_n\Big) < \rho_n\,.
\]
Putting together last two estimates and~(\ref{internal}) we obtain, for $n$ large enough,
\begin{equation}\label{psvac}
\H^N\bigg(\Big( E_n \Delta \big(Q \times (-R_n/2,\hat t_n) \big)\Big) \cap Q^{R_n}\bigg) < \frac \theta {2^{N+2}}\,.
\end{equation}
Now, for every $s<1$ we can call
\[
\Gamma_s = \Big(E_n \Delta \big(Q \times (-R_n/2,\hat t_n) \big)\Big) \cap \Big(\partial Q_s \times \big(-R_n/2,R_n/2\big)\Big)\,.
\]
By Fubini Theorem, the estimate~(\ref{psvac}) gives
\[
\frac \theta {2^{N+2}} > \int_{\frac 12}^1 \H^{N-1}\big(\Gamma_s\big)\,ds\,,
\]
and then there exists some $1/2<r_n^-<1$ such that
\begin{equation}\label{defrn-}
\H^{N-1}\big(\Gamma_{r_n^-}\big) < \frac \theta {2^{N+1}}\,.
\end{equation}
We define now the temptative set $\widehat F_n$ as
\[
\widehat F_n= \Big(E_n \setminus Q_{r_n^-}^{R_n}\Big) \cup \big(Q_{r_n^-} \times (-R_n/2,\hat t_n)\big)\,.
\]
According to~(\ref{defF}), the actual set $F_n$ will be defined in the very same way, just with some $\bar t$ in place of $\hat t$. First of all, we notice that
\begin{equation}\label{with2}
-\rho_n R_n\leq \hat t \leq\rho_n R_n\,,
\end{equation}
which is clear by construction since $\hat t_n\in [a_n,b_n]\subseteq [-R_n^-,R_n^+]\subseteq (-\rho_n R_n,\rho_n R_n)$. Moreover, we claim that
\begin{align}\label{fincltempt}
\H^N\big(E_n \Delta \widehat F_n\big)\leq \frac \theta {2^{N+2}} \,, && \H^{N-1}\big( \partial^* \widehat F_n \cap \partial Q_{r_n^-}^{R_n}\big)\leq \frac \theta{2^N}\,.
\end{align}
The left estimate is directly given by~(\ref{psvac}). Concerning the right one, notice that
\begin{equation}\label{subdiv}
\partial\big(Q_{r_n^-}^{R_n}\big)=
\big(\partial Q_{r_n^-} \times (-R_n/2,R_n/2)\big)\cup \big(Q_{r_n^-} \times \{-R_n/2\}\big)\cup \big(Q_{r_n^-} \times \{R_n/2\}\big)\,.
\end{equation}
By~(\ref{defrn-}), we have
\[
\H^{N-1}\Big(\partial^* \widehat F_n \cap \big(\partial Q_{r_n^-} \times (-R_n/2,R_n/2)\big)\Big) = \H^{N-1} \big(\Gamma_{r_n^-}\big)< \frac \theta {2^{N+1}}\,.
\]
Since $P_{\rm Eucl}(E_n; Q\times (-R_n/2,a_n))\leq 3\rho_n$ by~(\ref{star}), keeping in mind that by construction~(\ref{cond2}) is valid for $E_n$ with $\rho=\rho_n$ and $W=a_n$, by projection we deduce that
\[\begin{split}
\H^{N-1}\Big(\partial^* \widehat F_n &\cap \big(Q_{r_n^-} \times \{-R_n/2\}\big) \Big)
= \H^{N-1}\Big(\big(Q_{r_n^-}\times \{-R_n/2\}\big) \setminus E_n\Big)\\
&\leq \H^{N-1}\Big( \partial^* E_n \cap \big(Q_{r_n^-}\times (-R_n/2,a_n)\big)\Big)+\H^{N-1}\Big(\big(Q_{r_n^-}\times \{a_n\}\big)\setminus E_n\Big)\\
&\leq 3 \rho_n+\H^{N-1}\Big(\big(Q\times \{a_n\}\big)\setminus E_n\Big)\leq 4\rho_n\,,
\end{split}\]
and in the very same way
\[
\H^{N-1}\Big(\partial^* \widehat F_n \cap \big(Q_{r_n^-} \times \{R_n/2\}\big) \Big)= \H^{N-1}\Big(\big(Q_{r_n^-}\times \{R_n/2\}\big) \cap E_n\Big)\leq 4\rho_n\,.
\]
Putting together the last three estimates, by~(\ref{subdiv}) we obtain
\[
\H^{N-1}\Big( \partial^* \widehat F_n \cap \partial Q_{r_n^-}^{R_n}\Big)\leq \frac \theta {2^{N+1}} + 8\rho_n\,,
\]
and this implies that also the right estimate in~(\ref{fincltempt}) is true as soon as $n$ is large enough. Summarizing, the set $\widehat F_n$ already satisfies the first two estimates in~(\ref{fincl}), and we only have to modify it to get also the third. We can then finally define, according to~(\ref{defF}),
\begin{align*}
F_n= \Big(E_n \setminus Q_{r_n^-}^{R_n}\Big) \cup \big(Q_{r_n^-} \times (-R_n/2,\bar t_n)\big)\,, && \hbox{where} && \bar t_n = \hat t_n + \frac{\H^N(E_n)-\H^N(\widehat F_n) }{(r_n^-)^{N-1}}\,.
\end{align*}
First of all we underline that, thanks to the first estimate in~(\ref{fincltempt}), we have
\begin{equation}\label{hbcl}
\big| \bar t_n - \hat t_n\big| \leq \frac \theta{2^{N+2} (r_n^-)^{N-1}} \leq \frac \theta 8\,.
\end{equation}
As a consequence, thanks to~(\ref{with2}) we have that $\bar t \in [-\rho_n R_n-\theta,\rho_n R_n+\theta]$ as required, ad we have to check the validity of~(\ref{fincl}). The third estimate in~(\ref{fincl}) is satisfied by the definition of $\bar t$, so we have to check the first two. Concerning the first one, it holds since by construction and by~(\ref{fincltempt}) we have
\[\begin{split}
\H^N(E_n\Delta F_n) &\leq \H^N(E_n \Delta \widehat F_n) + \H^N(\widehat F_n\Delta F_n)
=\H^N(E_n \Delta \widehat F_n) + \Big|\H^N(\widehat F_n)-\H^N(E_n)\Big|\\
&\leq 2\H^N(E_n \Delta \widehat F_n)\leq \frac \theta{2^{N+1}}\leq \theta\,.
\end{split}\]
And finally, the second one is satisfied since by construction, by~(\ref{fincltempt}) and by~(\ref{hbcl}) we have
\[\begin{split}
\H^{N-1}\big(\partial^* F_n\cap \partial Q_{r_n^-}^{R_n}\big) &\leq 
\H^{N-1}\big( \partial^* \widehat F_n \cap \partial Q_{r_n^-}^{R_n}\big) + 2(N-1)(r_n^-)^{N-2} |\bar t_n- \hat t_n|\\
&\leq \frac\theta{2^N} + \frac{2(N-1)\theta}{2^{N+1}}\leq \theta\,.
\end{split}\]
The proof is then concluded.
\end{proof}

We are now almost ready to present the proof of Theorem~\mref{main2}. The last ingredient we need is an auxilary problem. More precisely, for any vector $\nu\in\S^{N-1}$, we consider clusters contained in $\R_{+,\nu}^N=\big\{ x\in\R^N,\, x\cdot \nu>0\big\}$, and we consider the relative perimeter $P_*(\cdot;\R^N_{+,\nu})$. This is defined as follows. For every set $F\subseteq\R^N_+$, we define
\[
P_*(F;\R^N_{+,\nu}) = \int_{\partial^* F\cap \R^N_{+,\nu}} \|\nu_F(x)\| \, d\H^{N-1}(x)\,,
\]
and as usual for every cluster $\F\subseteq\R^N_{+,\nu}$ we set
\[
P_*(\F;\R^N_{+,\nu}) = \frac 12\, \sum\nolimits_{j=0}^m P_*(F_j;\R^N_{+,\nu})\,.
\]
Notice that the part of the  boundary of a set $F$ or of a cluster $\F$ which lies on the hyperplane $\{x\cdot\nu=0\}$ has no cost. As a consequence, it is standard to observe that there are clusters $\F\subseteq\R^N_{+,\nu}$ such that $P_*(\F;\R^N_{+,\nu}) < J(|\F|)$. By compactness, there exists a constant $\xi>0$, only depening on the norm $\|\cdot\|$ (and then on $M$), such that for every $v\in(\R^+)^m$ with $|v|_1=1$ and for every $\nu\in\S^{N-1}$, there exists a cluster $\F^\nu_v\subseteq \R^N_{+,\nu}$ such that
\begin{align}\label{xiauxpro}
|\F^\nu_v|=v\,, && P_*(\F^\nu_v;\R^N_{+,\nu}) \leq (1-\xi) J(v)\,.
\end{align}
Notice that, as an obvious consequence, any such cluster must satisfy
\[
\H^{N-1}\Big( \partial^*\F^\nu_v \cap \big\{x\cdot \nu=0\big\}\Big) > \frac{\xi J(v)}{C_2^+}\,,
\]
since $P_*(\F^\nu_v)\geq J(v)$. Again by compactness, we can call
\begin{equation}\label{defd+}
d^+ = \max\Big\{ \diam(\F^\nu_v),\, |v|_1=1,\, \nu\in\S^{N-1}\Big\}\,.
\end{equation}
We are finally ready to present the proof of Theorem~\mref{main2}. The basic idea is that if a very small minimal cluster is not connected, then by Lemma~\ref{easycase} the biggest connected component must contain almost all the minimal cluster. Then, by making use of Proposition~\ref{flattening} we can make the boundary flat around some external point of the biggest component, and then substitute all the other connected components with a minimal cluster ``attached'' to the flat part as in the auxiliary problem just discussed.

\begin{proof}[Proof of Theorem~\mref{main2}]
We have to consider the case when the manifold $M$ is Riemannian (that is, the fixed norm is the Euclidean one). However, we will use this assumption almost only in Step~III and in the last Step~VI, so we write most of the proof for a generic fixed norm.\par

\step{I}{Definition of the clusters $\F_j$ and $\F$.}
Let us assume that the  claim is false. Then, we can take a sequence of minimal $m$-clusters $\{\E_j\}_{j\in\N}$ which are not connected and whose total volumes $\omega_j=\tv{\E_j}_M$ converge to $0$. We denote the volume of each of the clusters as $|\E_j|= \omega_j v_j$, so that each $v_j\in (\R^+)^m$ satisfies $|v_j|_1=1$. Moreover, we call $\E^+_j$ the largest connected component of $\E_j$, and we denote its volume as $|\E_j^+|=\omega_j v_j^-$; notice that $v_j^-\leq v_j$ (that is, for every $1\leq i \leq m$ we have $(v_j^-)_i\leq (v_j)_i$). There are now two possibilities: either $v_j-v_j^-$ converges to $0$, or it does not. In the second case, there exists some $\delta>0$ such that $\tv{\E_j^+}<(1-\delta)\tv{\E_j}$ for infinitely many indices $j\in\N$, but then Lemma~\ref{easycase} implies that $\tv{\E_j}\geq \eps_\delta$ for all these indices, and this is impossible since $\tv{\E_j}=\omega_j\to 0$. As a consequence, we can assume without loss of generality that $v_j-v_j^-\to 0$; in other words, the percentage of the total volume of the cluster $\E_j$ which corresponds to points not belonging to $\E_j^+$ goes to $0$ when $j\to\infty$.

By compactness, up to a subsequence we can assume the existence of a point $\bar p\in M$ and of a sequence $p_j\to \bar p$ such that each $p_j$ is contained in the cluster $\E_j^+$. Now, let $\sigma<4^{-N}$ be a small constant, whose exact value will be specified at the end, precisely in~(\ref{defsigma}). Let us fix a balanced atlas with oscillation smaller than $\sigma$; as already done before, without loss of generality we can assume the existence of a chart $\varphi_i : U_i\to V_ i$ such that $\bar p \in V_i$ and, calling $\hat x=\varphi_i^{-1}(\bar p)$,
\begin{equation}\label{choicepoint}
\| \cdot \|_{i, \hat x} = \|\cdot\| \qquad \hbox{or, equivalently,} \qquad h_i(\hat x, \nu) = \|\nu\| \quad \forall \, \nu \in \R^N\,.
\end{equation}

Applying Theorem~\ref{smalldiam}, we get a constant $\bar \eps_1$; then, for every $j$ large enough we have that $\omega_j < \bar\eps_1$, and then each connected component of the cluster $\E_j$ is contained in a single chart, and according to Corollary~\ref{corsmalldiam} its diameter is going to $0$. In particular, up to pass to a subsequence we can assume that all the clusters $\E_j^+$ are compactly contained in the chart $V_i$, with $\diam\big(\varphi_i^{-1}(\E_j^+)\big)\leq 4 C_4 \omega_j^{1/N}$. Now, we can select a sequence $\{x_j\}\subseteq\R^N$ so that each cluster
\begin{equation}\label{defone}
\F_j= x_j + \frac 1{\omega_j^{1/N}}\, \varphi_i^{-1}\big(\E_j^+\big)\,,
\end{equation}
contains the origin. Notice that the distance between $\hat x$ and $\varphi_i^{-1}(\E_j^+)$ goes to $0$.\par

By construction, the volume of the cluster $\F_j$ is $v_j^-$; moreover, each cluster $\F_j$ is contained in the ball centered at the origin and with radius $4 C_4$. As a consequence, up to extract a further subsequence, we can assume that the clusters $\F_j$ have a weak limit $\F$; this means that, writing $\F_j = \big(F_{j,1},\, F_{j,2},\, \dots\,,\, F_{j,m}\big)$, for every $1\leq q\leq m$ the characteristic function $\Chi{F_{j,q}}$ is converging to $\Chi{F_q}$ in the weak BV sense for $j\to \infty$. Notice that, since the clusters $\F_j$ are all contained in a single ball, the convergence is in particular strong in $L^1$; therefore, $|\F|=v$, where $v\in (\R^+)^m$ is the common limit of both $\{v_j\}$ and $\{v_j^-\}$, in particular $|v|_1=1$.\par

\step{II}{The validity of~(\ref{Fopt}).}
We claim now that
\begin{equation}\label{Fopt}
P_*(\F_j)\to P_*(\F)\,.
\end{equation}
Indeed, since by lower semicontinuity of the perimeter we know that $P_*(\F)\leq \liminf P_*(\F_j)$, if~(\ref{Fopt}) were false then there would be some positive constant $c$ such that, up to pass to a subsequence, $P_*(\F_j)>P_*(\F)(1+3c)$ for every $j\gg 1$. By continuity of the perimeter, for every $j\gg 1$ we could find some cluster $\widetilde \F_j$, very close to $\F$ but satisfying $|\widetilde \F_j|= v_j$, so that
\begin{equation}\label{absurd}
P_*(\F_j)>P_*(\widetilde \F_j)(1+2c)\,.
\end{equation}
Define now
\begin{equation}\label{deftwo}
\widetilde \E_j = \varphi_i \Big(\hat x + \omega_j^{1/N} \widetilde \F_j\Big)\,.
\end{equation}
Since $\hat x\in U_i$ and $\omega_j\to 0$, for every $j$ large enough the set $\hat x + \omega_j^{1/N} \widetilde \F_j$ is compactly contained in $U_i$, so the above definition makes sense. Keeping in mind~(\ref{choicepoint}), the definitions~(\ref{defone}) and~(\ref{deftwo}) ensure that when $j\to +\infty$ we have
\begin{align*}
\frac{P_M(\E_j^+)}{\omega_j^\frac{N-1}N\,P_*(\F_j)} \to 1\,, &&
\frac{P_M(\widetilde\E_j)}{\omega_j^\frac{N-1}N\,P_*(\widetilde \F_j)} \to 1 \,.
\end{align*}
However, these estimates together with~(\ref{absurd}) imply that for $j$ large enough
\[
P_M(\E_j) \geq P_M(\E_j^+) \geq (1+c) \,P_M(\widetilde\E_j)\,.
\]
Since $\widetilde\E_j$ is a cluster in $M$ with volume $|\widetilde\E_j|=\omega_j |\widetilde\F_j| = \omega_j v_j = |\E_j|$, the above estimate is a contradiction with the fact that $\E_j$ is a minimal cluster. Therefore, (\ref{Fopt}) is proved.

\step{III}{The point $\bar x$ and the inequalities~(\ref{new1})--(\ref{new3}).}
Let now $\bar x$ be a point in $\partial^*(\cup_{q=1}^m F_q)$; then, $\bar x$ belongs to the boundary of exactly one of the sets $F_q$ with $1\leq q\leq m$, say $q=\bar q$. Let $\nu\in\S^{N-1}$ be the outer normal vector to $F_{\bar q}$ at $\bar x$, and for every $r>0$ let us call $Q_r$ the $(N-1)$-dimensional cube centered at $\bar x$, with side $r$, orthogonal to $\nu$, and with sides parallel to an arbitrary basis of the hyperplane $\{\omega\in\R^N,\, \omega\cdot \nu =0\}$. In addition, for every $R>0$ we write
\begin{align*}
Q_r^{R,\pm} = \big\{ z \pm t \nu,\, z\in Q_r,\, 0\leq t<R/2\big\}\,, && Q_r^R = \big\{ z + t \nu,\, z\in Q_r,\, -R/2<t<R/2\big\} \,.
\end{align*}
Finally, let $\rho>0$ be a small number, to be specified in Step~V. By the blow-up properties of the boundary, if $R$ is small enough then we have
\begin{gather}
\H^N\big(F_{\bar q} \cap Q_R^{R.-}\big) \geq \bigg(\frac 12-\frac{\rho^2}9\bigg)\, R^N\,,\label{original1}\\
\H^N\Big(\bigcup\nolimits_{q=1}^m F_q \cap Q_R^{R,+}\Big) \leq \frac{\rho^2} 9\, R^N\,,\label{original2}\\
\H^{N-1}\Big(\partial^*\F \cap Q_R^R\Big)\leq \bigg(1+\frac\rho 5\bigg) R^{N-1}\,.\label{original3}
\end{gather}
We claim now that, choosing a suitable small $R$ such that~(\ref{original1})--(\ref{original3}) are in force, for $j$ large enough one has
\begin{gather}
\H^N\big(F_{j,\bar q} \cap Q_R^{R.-}\big) \geq \bigg(\frac 12-\frac {\rho^2} 8\bigg)\, R^N\,,\label{new1}\\
\H^N\Big(\bigcup\nolimits_{q=1}^m F_{j,q} \cap Q_R^{R,+}\Big) \leq \frac {\rho^2}8\, R^N\,,\label{new2}\\
\H^{N-1}\Big(\partial^*\F_j \cap Q_R^R\Big)\leq \bigg(1+\frac \rho 4\bigg) R^{N-1}\,.\label{new3}
\end{gather}
In fact, since the sets $F_{j,q}$ converge in the $L^1$ sense to $F_j$, the validity of~(\ref{new1}) and~(\ref{new2}) for $j$ large is an obvious consequence of~(\ref{original1}) and~(\ref{original2}). Instead, the lower semicontinuity of the perimeter implies that
\[
\H^{N-1}\Big(\partial^*\F \cap Q_R^R\Big) \leq \liminf \H^{N-1}\Big(\partial^*\F_j \cap Q_R^R\Big)\,,
\]
which does not allow to infer~(\ref{new3}) from~(\ref{original3}). However, we can choose a small $R$ such that not only~(\ref{original1})--(\ref{original3}) are true, but also
\begin{align}\label{nobound}
\H^{N-1}\Big(\partial^*\F \cap \partial\big(Q_R^R\big)\Big)=0\,, && \H^{N-1}\Big(\partial^*\F_j \cap \partial\big(Q_R^R\big)\Big)=0\quad \forall\, j\in\N\,.
\end{align}
This is not a problem since the above inequalities can fail for at most countably many $R>0$. Keeping in mind that all the clusters $\F_j$ are compactly contained in the ball $B_{5C_4}$ centered in the origin and with radius $5C_4$, the above mentioned lower semicontinuity of the perimeter implies that
\[
\H^{N-1}\Big(\partial^*\F \cap \Big( B_{5C_4}\setminus \overline{Q_R^R}\Big)\Big) \leq \liminf \H^{N-1}\Big(\partial^*\F_j \cap \Big( B_{5C_4}\setminus \overline{Q_R^R}\Big)\Big)\,.
\]
Keeping in mind~(\ref{Fopt}) from Step~II and~(\ref{nobound}), and using the fact that $M$ is a Riemannian manifold (so the fixed norm $\|\cdot\|$ coincides with the Euclidean norm $|\cdot|$) we deduce that
\[
\H^{N-1}\Big(\partial^*\F \cap Q_R^R\Big) \geq \limsup \H^{N-1}\Big(\partial^*\F_j \cap Q_R^R\Big)\,,
\]
and this allows to obtain~(\ref{new3}) from~(\ref{original3}). Summarizing, the estimates~(\ref{new1})--(\ref{new3}) are true with a fixed, small $R$, and for all $j$ large enough.

\step{IV}{Definition of $r$ and proof of~(\ref{othersmall}).}
Since $v_j-v_j^-\to 0$, as observed in Step~I, if $j$ is large enough then not only~(\ref{new1})--(\ref{new3}) are true, but also
\[
d^+ \big|v_j - v_j^-\big|_1^{1/N}< \frac R6\,,
\]
where $d^+$ is the constant defined in~(\ref{defd+}). We call then $2\leq n\in\N$ the number such that
\begin{equation}\label{setRr}
\frac R{3(n+1)} \leq d^+ \big |v_j - v_j^-\big|_1^{1/N} <  \frac R {3n}\,,
\end{equation}
where $|v|_1=\sum_{q=1}^m |v_q|$, and we set $r=R/n$; notice that $r$ depends on $j$, and actually $r\to 0$ when $j\to\infty$, but we do not write explicitely the dependance of $r$ on $j$ to keep the notation lighter. The $(N-1)$-dimensional cube $Q_R$ can be subdivided in $n^{N-1}$ essentially disjoint cubes $Q_{r,i}$ with side $r$, with $1\leq i \leq n^{N-1}$. Since $Q_R^R$ is essentially the union $\bigcup_{i=1}^{n^{N-1}} Q_{r,i}^R$, there is some $1\leq \bar i \leq n^{N-1}$ such that, writing for brevity $Q_r$ in place of $Q_{r,\bar i}$, one has
\begin{gather}
\H^N\big(F_{j,\bar q} \cap Q_r^{R.-}\big) \geq \bigg(\frac 12-\frac{\rho^2}2\bigg)\, r^{N-1}R\,,\label{final1}\\
\H^N\Big(\bigcup\nolimits_{q=1}^m F_{j,q} \cap Q_r^{R,+}\Big) \leq \frac{\rho^2} 2\, r^{N-1}R\,,\label{final2}\\
\H^{N-1}\Big(\partial^*\F_j \cap Q_r^R\Big)\leq (1+\rho) r^{N-1}\,.\label{final3}
\end{gather}
In fact, by~(\ref{new1}) we have that~(\ref{final1}) can fail for at most $25\%$ of the cubes $Q_{r,i}$, and the same is true for~(\ref{final2}) thanks to~(\ref{new2}) and for~(\ref{final3}) by~(\ref{new3}). Notice that~(\ref{final1}) and~(\ref{final2}) ensure that, writing for brevity
\[
\Fot= \bigcup\nolimits_{q\neq \bar q} F_{j,q}\,,
\]
one has
\[
\H^N\Big(\Fot  \cap Q_r^R\Big) \leq \rho^2 r^{N-1} R\,.
\]
But in fact, we claim that a much stronger inequality is true, namely,
\begin{equation}\label{othersmall}
\H^N\Big(\Fot  \cap Q_r^R\Big) \leq C_N' \rho r^N
\end{equation}
for some purely geometric constant $C_N'$, only depending on $N$. To show this estimate, for every $z\in Q_r$ we call $S_z=\{z+t\nu,\, -R<t<R\}$, and we call
\begin{gather*}
T_1 = \big\{ z \in Q_r,\, \partial^* F_{j,\bar q}\cap S_z=\emptyset \big\}\,,\\
T_2 = \big\{ z \in Q_r\setminus T_1,\, \#(\partial^* \F_j\cap S_z)\geq 2 \big\}\,,\\
T_3 = \big\{ z \in Q_r\setminus (T_1\cup T_2),\, \cup_{q\neq \bar q}\,\partial^* F_{j,q}\cap S_z\neq \emptyset \big\}\,.
\end{gather*}
If $z\in T_1$, then $S_z$ either does not intersect $F_{j,\bar q}$, or is entirely contained in it. Therefore, by~(\ref{final1}) and~(\ref{final2})
\[
\rho^2 r^{N-1} R \geq \H^N\big(F_{j,\bar q}\Delta Q_r^{R,-}\big)\geq \frac R2\, \H^{N-1}(T_1)\,,
\]
and this implies
\[
\H^{N-1}(T_1)\leq 2\rho^2 r^{N-1}\,.
\]
Moreover, by~(\ref{final3})
\[
(1+\rho)r^{N-1}\geq \H^{N-1}\Big(\partial^*\F_j \cap Q_r^R\Big) \geq \H^{N-1}(Q_r\setminus T_1)+\H^{N-1}(T_2)
\geq (1-2\rho^2) r^{N-1} +\H^{N-1}(T_2)\,,
\]
and this implies
\[
\H^{N-1}(T_2)\leq (\rho+2\rho^2) r^{N-1}\,.
\]
Finally, if $z\in T_3$ then $S_z$ contains only a single point of $\partial\F_j$, and this point is common boundary between $F_{j,\bar q}$ and exactly one $F_{j,q}$ with $q\neq \bar q$. As a consequence, the whole $S_z$ is contained in $\cup_{q=1}^m F_{j,q}$, and then by~(\ref{final2}) we have
\[
\frac{\rho^2}2 \, r^{N-1} R \geq \H^N\Big(\bigcup\nolimits_{q=1}^m F_{j,q} \cap Q_r^{R,+}\Big)\geq R \H^{N-1}(T_3)\,,
\]
which implies
\[
\H^{N-1}(T_3) \leq \frac{\rho^2}2\, r^{N-1}\,.
\]
The bounds on the measure of $T_1,\, T_2$ and $T_3$, and the fact that $\rho$ is small, ensure that
\begin{equation}\label{Tsmall}
\H^{N-1}(T) =\H^{N-1}(T_1\cup T_2\cup T_3) \leq \bigg(\rho +\frac 92\, \rho^2\bigg) r^{N-1} \leq 2\rho r^{N-1}\,,
\end{equation}
where we have called $T=T_1\cup T_2\cup T_3$. Notice now that if $z\notin T$ then $S_z$ contains exactly a point of $\partial^*\F_j$, and this point belongs to $\partial^* F_{j,\bar q}$ and to no other $\partial^* F_{j,q}$ with $q\neq \bar q$. Therefore,
\begin{equation}\label{torepeat}
\H^{N-1}\Big( \partial^*\F_j \cap \big((Q_r\setminus T)\times (-R/2,R/2)\big)\Big)\geq \H^{N-1}(Q_r\setminus T) \geq \big(1 - 2\rho\big) r^{N-1}\,.
\end{equation}
By definition, and keeping in mind~(\ref{final3}), we have then
\begin{equation}\label{useV}\begin{split}
\hspace{25pt}&\hspace{-25pt}\H^{N-1}\bigg( \partial^* \Fot\cap Q_r^R\bigg)
=\H^{N-1}\bigg( \partial^* \Fot\cap  \big( T\times(-R/2,R/2)\big)\bigg)\\
&\leq \H^{N-1}\Big(\partial^* \F_j \cap \big(T \times (-R/2,R/2)\big)\Big)\\
&=\H^{N-1}\Big(\partial^* \F_j \cap Q_r^R\Big) - \H^{N-1}\Big(\partial^* \F_j \cap \big((Q_r\setminus T)\times (-R/2,R/2)\big)\Big)
\leq 3\rho r^{N-1}\,.
\end{split}\end{equation}
Since $\rho$ is very small, by the relative isoperimetric inequality the above bound implies
\begin{equation}\label{useV'}
\H^{N-1}\Big( \partial^* \big(\Fot  \cap Q_r^R\big)\Big) \leq C_N \rho r^{N-1}
\end{equation}
for some purely dimensional constant $C_N$ (which is elementary to calculate, but whose exact value plays no role here). And then, by the standard isoperimetric inequality we deduce
\[
\H^N\Big(\Fot  \cap Q_r^R\Big) \leq \frac{\big(C_N \rho r^{N-1}\big)^{\frac N{N-1}}}{N^{\frac N{N-1}} \omega_N^{\frac 1{N-1}}}
=C_N' \rho^{\frac N{N-1}} r^N\leq C_N' \rho r^N
\]
with a suitable, purely dimensional constant $C_N'$. This ensures the validity of~(\ref{othersmall}), and thus concludes this step.

\step{V}{Definition of the cluster $\F_j'$ and validity of~(\ref{smallness}).}
Now, we can specify the value of the constant $\rho$. Namely, we set $\theta=\sigma$, and we choose $\rho$ corresponding to $\theta$ by means of Proposition~\ref{flattening}, possibly reduced so that
\begin{align}\label{choicerho}
C_N' \rho \leq \frac 1{(6d^+)^N}\,, && 2^{N-1} (C_N')^{\frac{N-1}N}  N\omega_N^{1/N} m^{1/N} \rho^{\frac{N-1}N}\leq \sigma\,,
\end{align}
where $d^+$ is the constant given by~(\ref{defd+}). We aim now to apply Proposition~\ref{flattening} to the set $F_{j,\bar q}$; notice that this is possible because~(\ref{hyp1}) is valid by~(\ref{final3}), since $\partial^* F_{j,\bar q}\subseteq \partial^* \F_j$, and~(\ref{hyp2}) is valid by~(\ref{final1}) and~(\ref{final2}). We obtain then $\bar t\in [-\rho R-\theta r, \rho R+\theta r]$ and $r/2\leq r^-\leq r$ such that, calling
\[
F'_{j,\bar q}= \big(F_{j,\bar q} \setminus Q_{r^-}^R\big) \cup \big\{ z + t \nu,\, z\in Q_{r^-},\, -R/2<t<\bar t\big\}\,,
\]
one has
\begin{align}\label{from4.4}
\H^N\big(F_{j,\bar q} \Delta F_{j,\bar q}' \big)\leq \theta r^N\,, && \H^{N-1}\Big( \partial^* F_{j,\bar q}' \cap \partial Q_{r^-}^R\Big)\leq \theta\, r^{N-1}\,, && \H^N(
F_{j,\bar q}) = \H^N(F_{j,\bar q}')\,.
\end{align}
We can now define a modified cluster $\F_j'$; namely, while $F_{j,\bar q}'$ has been just defined, for every $q\in \{1,\, 2,\, \dots\,,\, m\}\setminus\{\bar q\}$ we set $F_{j,q}'=F_{j,q}\setminus Q_{r^-}^R$. We also call $v_j'=|\F_j'|$. Notice that, since for every $q\neq \bar q$ we have
\[
|v_{j,q}^- - v_{j,q}'| = \big| |F_{j,q}| - |F_{j,q}'|\big| = \big| F_{j,q} \cap Q_{r^-}^R\big|\,,
\]
by~(\ref{othersmall}) we get
\begin{align}\label{smallness}
\sum_{q\neq \bar q} |v_{j,q}^- - v_{j,q}'|
= \H^N\Big(\Fot \cap Q_{r^-}^R\Big)
\leq C_N' \rho r^N\leq 2^N C_N' \rho (r^-)^N\,, && \hbox{while} && v_{j,\bar q}'=v_{j,\bar q}^-\,.
\end{align}
Notice that by construction, keeping in mind that $\bar t+r\leq \rho R+(\theta+1) r<R$ since $\theta$ and $\rho$ are small and $R=nr\geq 2r$, we have
\begin{equation}\label{emptyint}
\F_j' \cap \big(Q_{r^-}\times (\bar t,\bar t + r)\big)
\subseteq \F_j' \cap \big(Q_{r^-}\times (\bar t,R/2)\big)=\emptyset\,.
\end{equation}
We conclude this step by showing that
\begin{gather}
\H^{N-1}\Big(\partial^* \F_j'\cap Q_{r^-}^R\Big)= (r^-)^{N-1}\,,\label{useVI}\\
\H^{N-1}\Big(\partial^* F_{j,\bar q}\cap Q_{r^-}^R\Big)\geq \big(1-2^N\rho\big) (r^-)^{N-1}\,,\label{useVI'}\\
\H^{N-1}\Big( \partial^* \big(\Fot  \cap Q_{r^-}^R\big)\Big) \leq 2^{N-1} C_N \rho (r^-)^{N-1}\,.\label{useVI''}
\end{gather}
The property~(\ref{useVI}) is true by construction. The property~(\ref{useVI'}) can be obtained arguing as in~(\ref{torepeat}); indeed, keeping in mind~(\ref{Tsmall}) and since $r\leq 2r^-$, we have
\[\begin{split}
\H^{N-1}\Big(\partial^* F_{j,\bar q}\cap Q_{r^-}^R\Big) &\geq
\H^{N-1}\Big( \partial^* F_{j,\bar q} \cap \big((Q_{r^-}\setminus T)\times (-R/2,R/2)\big)\Big)\geq \H^{N-1}(Q_{r^-}\setminus T)\\
&\geq (r^-)^{N-1}-\H^{N-1}(T)
\geq (r^-)^{N-1}-2\rho r^{N-1}\geq \big(1 - 2^N\rho\big) (r^-)^{N-1}\,.
\end{split}\]
To obtain~(\ref{useVI''}), we first write
\[\begin{split}
\H^{N-1}\Big(\partial^* \F_j\cap \big(Q_r^R\setminus Q_{r^-}^R\big)\Big) &\geq \H^{N-1}\big((Q_r\setminus Q_{r^-})\setminus T\big)\\
&=r^{N-1}-(r^-)^{N-1}-\H^{N-1}\big(T\cap(Q_r\setminus Q_{r^-})\big)\,.
\end{split}\]
Hence, by~(\ref{final3}) we get
\[\begin{split}
\H^{N-1}\Big(\partial^*\F_j\cap Q_{r^-}^R\Big) &= \H^{N-1}\Big(\partial^*\F_j\cap Q_r^R\Big) - \H^{N-1}\Big(\partial^*\F_j\cap \big(Q_r^R\setminus Q_{r^-}^R\big)\Big)\\
&\leq \rho r^{N-1}+(r^-)^{N-1}+\H^{N-1}\big(T\cap(Q_r\setminus Q_{r^-})\big)\,.
\end{split}\]
Arguing again as in~(\ref{torepeat}), we have
\[
\H^{N-1}\bigg(\partial^* \F_j \cap \bigg((Q_{r^-}\setminus T)\times \bigg(-\frac R2,\frac R2\bigg)\!\bigg)\!\bigg)\geq \H^{N-1}(Q_{r^-}\setminus T) = (r^-)^{N-1} - \H^{N-1}(T \cap Q_{r^-})\,.
\]
Arguing as in~(\ref{useV}) and again keeping in mind~(\ref{Tsmall}) and the fact that $r\leq 2r^-$, the last two estimates give then
\[\begin{split}
\H^{N-1}\bigg( &\partial^* \Fot\cap Q_{r^-}^R\bigg)\\
&\leq \H^{N-1}\Big(\partial^* \F_j \cap Q_{r^-}^R\Big) - \H^{N-1}\bigg(\partial^* \F_j \cap \bigg((Q_{r^-}\setminus T)\times \bigg(-R/2,R/2\bigg)\bigg)\bigg)\\
&\leq \rho r^{N-1}+\H^{N-1}(T)\leq 3\rho r^{N-1} \leq 3\cdot 2^{N-1} \rho (r^-)^{N-1}\,.
\end{split}\]
And finally, exactly as in~(\ref{useV'}), by the relative isoperimetric inequality the last bound implies~(\ref{useVI''}).

\step{VI}{Definition of $\E_j'$ and $\E_j''$ and conclusion.}
We are finally in position to conclude. Keeping in mind that $\F_j$ was defined by $\E_j^+$ by~(\ref{defone}), we would like to set
\begin{equation}\label{defej'}
\E_j'=\varphi_i\Big(\omega_j^{1/N} \big(\F_j'- x_j\big) \Big)\,.
\end{equation}
Let us check that this definition makes sense. First of all, we observe that by construction $\F_j\Delta \F_j'$ is contained in $Q_R^R$, which is a cube with side $R$ which surely intersects $\F_j$, for instance by~(\ref{new1}). Therefore, calling for a moment $\G_j=\omega_j^{1/N} \big(\F_j'- x_j\big)$, we have that $\varphi_i^{-1}(\E_j^+)\Delta \G_j$ is contained in a cube of side $\omega_j^{1/N} R$ which intersects $\varphi_i^{-1}(\E_j^+)$. As shown in Step~I, the sets $\varphi_i^{-1}(\E_j^+)$ have diameters which go to $0$ and they converge to a point $\hat x\in U_i$. Since $U_i$ is open and $\omega_j\to 0$, this implies that $\G_j$ is compactly contained in $U_i$ for every $j$ large enough. Therefore, the cluster $\E_j'=\varphi_i(\G_j)$ is well-defined as soon as $j$ is large enough, and then the definition~(\ref{defej'}) makes sense. Moreover, by the construction of Step~V we know that for every $q\neq \bar q$ the set $F_{j,q}'$ is contained in $F_{j,q}$, thus minding~(\ref{defone}) and~(\ref{defej'}) we have also that $E_{j,q}'$ is contained in $E_{j,q}^+$; in addition, by~(\ref{smallness})
\[
\sum_{q\neq \bar q} \big| E^+_{j,q}\setminus E_{j,q}'  \big|
= \omega_j \H^N\Big(\bigcup\nolimits_{q\neq \bar q} F_{j,q}  \cap Q_{r^-}^R\Big)
\leq C_N' \omega_j \rho r^N\,.
\]
We now want to estimate $P_M(\E_j')-P_M(\E_j)$. In doing this, we will again use that the manifold is Riemannian, and not just of fixed type. Since the atlas is balanced and has oscillation smaller than $\sigma$, then $1-\sigma\leq h_i(x,\nu)\leq 1+\sigma$ for every $x\in U_i$ and $\nu\in\S^{N-1}$. By comparing~(\ref{defej'}) and~(\ref{defone}), and keeping in mind the construction of $\F_j'$, we have then
\[\begin{split}
\frac{P_M(\E_j')-P_M(\E_j^+)}{\omega_j^{\frac{N-1}N}} \leq (1+\sigma) \H^{N-1}\Big(\big(\partial^* F'_{j,\bar q}\cap \overline{Q_{r^-}^R}\big)&\cup \partial^* \big(\Fot  \cap Q_{r^-}^R\big)\Big)\\
&-(1-\sigma)\H^{N-1}\big(\partial^* F_{j,\bar q}\cap \overline{Q_{r^-}^R}\big)\,.
\end{split}\]
Putting together~(\ref{useVI''}), (\ref{useVI}), and the second inequality in~(\ref{from4.4}), we have
\[\begin{split}
\H^{N-1}\Big(\big(\partial^* &F'_{j,\bar q}\cap \overline{Q_{r^-}^R}\big)\cup \partial^* \big(\Fot  \cap Q_{r^-}^R\big)\Big)\\
&\leq \H^{N-1}\Big(\partial^* F'_{j,\bar q}\cap Q_{r^-}^R\Big) +\H^{N-1}\Big(\partial^* F'_{j,\bar q}\cap \partial Q_{r^-}^R\Big) + 2^{N-1} C_N \rho (r^-)^{N-1}\\
&\leq (r^-)^{N-1}+ \theta r^{N-1} + 2^{N-1} C_N \rho (r^-)^{N-1}
\leq \Big( 1 + 2^{N-1}\big(\theta + C_N \rho\big)\Big) \, (r^-)^{N-1}\,,
\end{split}\]
while~(\ref{useVI'}) gives
\[
\H^{N-1}\Big(\partial^* F_{j,\bar q}\cap \overline{Q_{r^-}^R}\Big)\geq \big(1-2^N\rho\big) (r^-)^{N-1}\,,
\]
and then the above inequality implies
\begin{equation}\label{iniz}\begin{split}
\frac{P_M(\E_j')-P_M(\E_j^+)}{\omega_j^{\frac{N-1}N}} &\leq 
\Big((1+\sigma) \big( 1 + 2^{N-1}\big(\theta + C_N \rho\big)\big) 
-(1-\sigma) \big(1-2^N\rho\big) \Big)(r^-)^{N-1}\\
&\leq \Big(2\sigma + 2^N\big(\theta+(C_N+1)\rho\big)\Big) (r^-)^{N-1}
\leq 2^N(C_N+3)\sigma (r^-)^{N-1}\,,
\end{split}\end{equation}
where the last inequality is true since by construction $\rho\leq \theta=\sigma$. We want to define a competitor $\E_j''$ for $\E_j$. Keeping in mind~(\ref{defej'}), we will simply set
\begin{equation}\label{defej''}
\E_j''=\varphi_i\Big(\omega_j^{1/N} \big(\F_j''- x_j\big) \Big)\,,
\end{equation}
where $\F_j''=\F_j' \cup \widetilde \F$ and $\widetilde \F$ will be a suitable cluster with empty intersection with $\F_j'$. Since the volume of $\E_j''$ must equal that of $\E_j$, and since $|\E_j|=\omega_j v_j$ and $|\E_j'|=\omega_j v_j'$, we must have $|\widetilde \F|= v_j-v_j'$. We define now the cluster $\widetilde\F$ as a minimizer of the relative perimeter $P_*(\cdot;\R^N_{+,\nu})$ among all the clusters of volume $v_j-v_j'$; keep in mind that $\R^N_{+,\nu}$ is the half-space with outer normal equal to $\nu$, and that $\nu$ is the normal vector to $F_{\bar q}$ at $\bar x$. By~(\ref{xiauxpro}), we know that
\begin{equation}\label{aate}
P_*(\widetilde\F;\R^N_{+,\nu}) \leq (1-\xi) J\big(v_j-v_j'\big)\,.
\end{equation}
Let us then call $\Gamma=\partial^* \widetilde\F \cap \R^N_{+,\nu}$, so that $P_*(\widetilde\F;\R^N_{+,\nu})=\H^{N-1}(\Gamma)$. Thanks to~(\ref{defd+}), keeping in mind that $r=R/n$ where $n\geq 2$ has been defined by~(\ref{setRr}), and by~(\ref{smallness}), we have then
\[
\diam(\widetilde \F) \leq d^+ |v_j-v_j'|_1^{1/N}
\leq d^+ |v_j-v_j^-|_1^{1/N} + d^+ |v_j^--v_j'|_1^{1/N}
\leq \frac r 3 + 2 d^+ (C_N' \rho)^{1/N}  r^- \leq r^-\,,
\]
where the last inequality is true thanks to~(\ref{choicerho}), left.\par

This diameter estimate has a crucial consequence. Indeed, this implies that the ``flat part'' of $\partial\widetilde\F$ is contained in $Q_{r^-}\times \{\bar t\}$. Let us be more precise. Thanks to the diameter estimate, we can translate the cluster $\widetilde \F$ (without changing its name, to keep the notation as clean as possible) in such a way that $\partial^*\widetilde\F\setminus\Gamma \subseteq Q_{r^-}\times \{\bar t\}$, and that $\widetilde\F$ is contained in the cube $Q_{r^-}\times [\bar t,\bar t+ r^-]$. Since by~(\ref{emptyint}) this cube has no intersection with $\F_j'$, we have that $\widetilde\F\cap\F_j'=\emptyset$ as required; thus, the cluster defined in~(\ref{defej''}) has the same volume as $\E_j$, so it is a possible competitor. Now, notice carefully that, by construction, the ``flat part'' $\partial^*\widetilde\F\setminus\Gamma$ is contained in $\partial^*\F_j'$. As a consequence, the difference $P_M(\E_j'')-P_M(\E_j')$ is only given by the ``new boundary'' coming from $\Gamma$. More precisely, using again that the oscillation of the atlas is less than $\sigma$ and~(\ref{aate}), we have
\[
\frac{P_M(\E_j'')-P_M(\E_j')}{\omega_j^{\frac{N-1}N}} \leq (1+\sigma) \H^{N-1}(\Gamma)
\leq (1+\sigma) (1-\xi) J\big(v_j-v_j'\big)\,.
\]
Now, since the functional $J$ is subadditive, we have
\[\begin{split}
J(v_j-v_j') &\leq J(v_j - v_j^-) + J(v_j^--v_j')
\leq J(v_j - v_j^-) + N\omega_N^{1/N} m^{1/N} |v_j^--v_j'|_1^{\frac{N-1}N}\\
&\leq J(v_j - v_j^-) + \sigma (r^-)^{N-1}\,,
\end{split}\]
where the last inequality is true thanks to~(\ref{smallness}) and~(\ref{choicerho}), right. Putting the last two estimates together with~(\ref{iniz}), we have then
\begin{equation}\begin{split}\label{wirkend}
\frac{P_M(\E_j'')-P_M(\E_j^+)}{\omega_j^{\frac{N-1}N}} &\leq 2^N(C_N+3)\sigma (r^-)^{N-1}+(1+\sigma) (1-\xi) \Big(J(v_j - v_j^-) + \sigma (r^-)^{N-1}\Big)\\
&\leq 2^N(C_N+4)\sigma (r^-)^{N-1}+(1+\sigma-\xi) J(v_j - v_j^-)\,.
\end{split}\end{equation}
The last observation to be done is that $\E_j$ is the union of finitely many connected components; the largest is $\E_j^+$, and all the others, together, have volume $\omega_j(v_j-v_j^-)$. Therefore, since the atlas is balanced and with oscillation smaller than $\sigma$, and each connected component of $\E_j$ is compactly contained in a chart, we have
\[
P_M(\E_j)= P_M(\E_j^+) + P_M(\E_j\setminus \E_j^+) \geq P_M(\E_j^+) + (1-\sigma)\omega_j^{\frac{N-1}N} \, J(v_j-v_j^-)\,.
\]
This estimate and~(\ref{wirkend}) imply that
\begin{equation}\label{wirkond}
\frac{P_M(\E_j'')-P_M(\E_j)}{\omega_j^{\frac{N-1}N}} \leq 2^N(C_N+4)\sigma (r^-)^{N-1}-(\xi-2\sigma) J(v_j - v_j^-)\,.
\end{equation}
At last, we notice the obvious estimate
\[
J(v) \geq N\omega_N^{1/N} |v|_1^{\frac{N-1}N}\,,
\]
which in particular, again recalling that $r=R/n$ with $n\geq 2$ identified by~(\ref{setRr}), gives
\[
J(v_j-v_j^-) \geq N\omega_N^{1/N} |v_j-v_j^-|_1^{\frac{N-1}N}\geq N\omega_N^{1/N} \frac {(r^-)^{N-1}}{(5d^+)^{N-1}}\,.
\]
We can then finally specify the value of $\sigma$. Indeed, we can take $\sigma<4^{-N}$ as any constant satisfying
\begin{equation}\label{defsigma}
2^N(C_N+4)\sigma < (\xi-2\sigma)\, \frac {N\omega_N^{1/N}}{(5d^+)^{N-1}} \,.
\end{equation}
Notice that such constant $\sigma$ is defined only depending on the dimension $N$. Since this inequality, inserted in~(\ref{wirkond}), gives $P_M(\E_j'')< P_M(\E_j)$, we have finally obtained the required contradiction, and the thesis is concluded.
\end{proof}

\section{Asymptotic Equivalence of the Isoperimetric Profiles and final remarks\label{secfin}}

In this last section we give a proof of Theorem~\ref{AsymptoticTheorem} and present some final remarks. In fact, Theorem~\ref{AsymptoticTheorem} is a simple consequence of Theorem~\ref{smalldiam}, without making use of Theorems~\mref{main1} or~\mref{main2}.

\begin{proof}[Proof of Theorem~\ref{AsymptoticTheorem}]
We directly consider the general case of a Finsler manifold and prove property~(\ref{eq12bis}), since then the Riemannian case and property~(\ref{eq12}) are just a particular case, keeping in mind that for a Riemannian manifold one has $C_2^-=C_2^+=1$.

Let us start by calling $d_{\R^N}$ the maximal diameter for $\R^N$ with Euclidean metric, according to Definition~\ref{defmadi}. Then, we fix a small number $\sigma>0$, and we take any balanced atlas with oscillation smaller than $\sigma$; moreover, we call $\hat r$ the largest radius of the charts $U_i$ (which are finitely many balls in $\R^N$). Let us apply Theorem~\ref{smalldiam}, obtaining a number $\bar\eps_1$; up to replace it with a smaller constant, we can assume that $\bar\eps_1^{1/N} d_{\R^N} < \hat r$.\par

Let now $\F\subseteq\R^N$ be any cluster with total volume $\|\F\|<\bar\eps_1$, and which is optimal in $\R^N$ with the Euclidean metric, thus $P_{\rm Eucl}(\F)=J_{\R^N}(v)$ where $v=|\F|$ and by $P_{\rm Eucl}$ we denote the standard Euclidean perimeter of sets or clusters in $\R^N$. By definition, the diameter of $\F$ is less than $d_{\R^N} |v|_1^{1/N}< \hat r$; thus, up to a translation, we can assume that $\F$ is compactly contained in some $U_i$, and then there is some point $x_i\in U_i$ such that the unit ball $\{h_i(x_i,\cdot)\leq 1\}$ is balanced. Let us then call $\E = \varphi_i(\F)$, which is a cluster in $M$ with volume $|\E|_M=v$. For any $1\leq j\leq m$, keeping in mind the definition~(\ref{defC_2^+}) of $C_2^+$ and Definition~\ref{defosc} of the oscillation, we have
\begin{equation}\label{pocofa}\begin{split}
P_M(E_j) &= \int_{\partial^* F_j} h_i\big(x,\nu_{F_j}(x)\big)\,d\H^{N-1}(x)\\
&\leq (1+\sigma) \int_{\partial^* F_j} h_i\big(x_i, \nu_{F_j}(x)\big)\,d\H^{N-1}(x)
\leq (1+\sigma) C_2^+ P_{\rm Eucl}(F_j) \,.
\end{split}\end{equation}
Since of course the same estimate is valid with $\cup_{j=1}^m E_j$, recalling that the perimeter of the clusters is given by~(\ref{perclu}) we deduce
\[
J_M(v) \leq P_M(\E) \leq (1+\sigma) C_2^+ P_{\rm Eucl}(\F) = (1+\sigma) C_2^+ J_{\R^N} (v)\,.
\]
Conversely, let $\E$ be an optimal cluster in the manifold $M$ having total volume $\|\E\|\leq \bar\eps_1$, and call $v=|\E|$. Lemma~\ref{fmcc} ensures that $\E$ is done by finitely many connected components, say $\E=\bigcup_{\ell=1}^H \E_\ell$ for some $H\in\N$. Moreover, by Theorem~\ref{smalldiam} we know that each connected component is compactly contained in a chart, so for every $1\leq \ell \leq H$ there exists $i(\ell)$ such that $\E_\ell \comp V_{i(\ell)}$. Calling then $\F_\ell = \varphi_{i(\ell)}^{-1}(\E_\ell)$, arguing exactly as in~(\ref{pocofa}) using this time the definition~(\ref{defC_2^-}) of $C_2^-$, we obtain
\[
P_M(\E_\ell) \geq \frac {1-\sigma}{C_2^-}\, P_{\rm Eucl}(\F_\ell)\,.
\]
We can then call $\F\subseteq\R^N$ the union of the different clusters $\F_\ell$, that we can assume to be disjoint up to translations. So, adding over $\ell$, we get
\[
J_M(v) = P_M(\E) = \sum_{\ell=1}^H P_M(\E_\ell) 
\geq \frac {1-\sigma}{C_2^-}\, \sum_{\ell=1}^H  P_{\rm Eucl}(\F_\ell)
= \frac {1-\sigma}{C_2^-}\, P_{\rm Eucl}(\F) 
\geq \frac {1-\sigma}{C_2^-}\, J_{\R^N}(v) \,.
\]
Summarizing, for every $\sigma>0$ we have found $\bar\eps_1$ such that, for every $v\in(\R^+)^m$ satisfying $|v|\leq \bar\eps_1$, one has
\[
\frac {1-\sigma}{C_2^-}\leq \frac{J_M(v)}{ J_{\R^N}(v) } \leq (1+\sigma) C_2^+ \,,
\]
and this concludes the proof of~(\ref{eq12bis}).
\end{proof}

\subsection{Final remarks\label{afn}}

We conclude now with a few important remarks. The first one deals with an immediate extension of Theorem~\ref{AsymptoticTheorem}.

\begin{remark}
In Theorem~\ref{AsymptoticTheorem} we have considered the case of a Riemannian manifold and the case of a Finsler manifold. In order to keep the claim short, we have not explicitely considered the ``intermediate'' case of a fixed-norm manifold (see Definition~\ref{fnm}). However, it is clear that the very same proof gives also that, if $M$ is fixed-norm with norm $\|\cdot\|$, then
\[
\frac{J_M(v)}{J^{\|\cdot\|}_{\R^N}(v)} \to 1\qquad \hbox{as}\ |v|\to 0\,,
\]
where by $J_{\R^N}^{\|\cdot\|}$ we denote the multi-isoperimetric profile of $\R^N$ endowed with the norm $\|\cdot\|$.
\end{remark}

The second remark deals with another immediate extension, this time of Theorem~\mref{main2}.

\begin{remark}\label{easy}
We can consider the natural notion of an ``almost Riemannian'' manifold as follows. We say that the manifold $M$ is \emph{almost Riemannian (with parameter $\delta$)} if there is a constant $\delta>0$ such that for any $p\in M$ and any $\nu\in\R^N$ one has that
\[
1-\delta\leq \frac{\|\nu\|_p}{|\nu|} \leq 1+\delta\,.
\]
Of course this property is interesting only for a small value of the parameter $\delta$, and a manifold is Riemannian if and only if it is almost Riemannian with parameter $\delta$ for every positive number $\delta$. A quick look at the proof of Theorem~\mref{main2} ensures that the claim is true not only for a Riemannian manifold, but also for a manifold which is almost Riemannian with sufficiently small parameter. In fact, the fact that $M$ was a Riemannian manifold was only used to be sure that, inside a chart $U_i$ of a balanced atlas of oscillation $\sigma$, one has $1-\sigma\leq h_i(x,\nu)\leq 1+\sigma$ for every $x\in U_i$ and $\nu\in\S^{N-1}$. However, the same inequality is clearly true also if the manifold is almost Riemannian with a parameter $\delta<\sigma/3$ and provided that we use a balanced atlas with oscillation $\sigma/3$; then the very same proof works also in this more general case.
\end{remark}

The next remark ensures that a full extension of Theorem~\mref{main2} is false.

\begin{remark}
It could seem reasonable to guess that Theorem~\mref{main2} remains true for any Finsler manifold, but this is actually false. In fact, thanks to a counterexample contained in~\cite{CarazzatoPratelli2025}, there exists a compact Finsler manifold inside which there are non connected minimal clusters with arbitrarily small total volume. However, in that example $m$ is quite large and the connected components are just two, so this leaves space for further investigation. In particular, it is not clear to us whether the connectedness of small minimal clusters in generic compact Finsler manifolds is true in the case $m=2$.
\end{remark}

The last remark deals again with a possible extension of Theorem~\mref{main2}.

\begin{remark}\label{lastremark}
We believe that Theorem~\mref{main2} remains true for the more general case of a fixed-norm manifold (and then, arguing as in Remark~\ref{easy}, for any ``almost fixed-norm manifold''). In fact, in most of the construction of Section~\ref{secma2} we were able to use a generic fixed-norm manifold, and not necessarily a Riemannian one. The assumption that $M$ is actually Riemannian has been used only in the proof of Theorem~\mref{main2}, in few (but important) points.
\end{remark}

\section*{Acknowledgemnts}
The authors acknowledge the hospitality of the Universidade Federal do ABC and of the University of Pisa. The work of SN was partially supported by FAPESP Aux\~alio Jovem Pesquisador \#2021/05256-0, CNPq  Bolsa de Produtividade em Pesquisa 1D \#12327/2021-8, 23/08246-0, Geometric Variational Problems in Smooth and Nonsmooth Metric Spaces \#441922/2023-6. The work of AP was partially supported by the PRIN 2022 project \#2022E9CF89.

\bibliographystyle{abbrv}
\bibliography{bibliography}

\end{document}